\documentclass[12pt]{amsart}

\usepackage{graphicx}    % needed for including graphics e.g. EPS, PS
\usepackage{amssymb,amsmath,latexsym,amsbsy,amstext,amsthm}
\usepackage{amsthm}
\usepackage{hyperref}
\usepackage{caption}
\usepackage{subcaption}

\usepackage{tikz} % for picture
\newtheorem{theorem}{Theorem}[section]
\newtheorem{lemma}[theorem]{Lemma}

\newtheorem{corollary}[theorem]{Corollary}
\newtheorem{proposition}[theorem]{Proposition}

\newcommand{\e}{\epsilon}

\newcommand{\g}{\gamma}

\newcommand{\I}{\mbox{Im\,}}

\newcommand{\lbda}{\lambda}
\newcommand{\mb}{\mathbb}

\newcommand{\R}{\mbox{Re\,}}

\newcommand{\w}{\omega}

\renewcommand{\a}{\alpha}
\renewcommand{\l}{\lambda}
\renewcommand{\H}{\mathbb H}
\newcommand{\G}{\Gamma}

\begin{document}      
\title{Regularity of Loewner curves}         % Enter your title between curly braces
\date{Preprint - \today}          % Enter your date or \today between curly braces
\author{Joan Lind}\thanks{Lind's work partially supported by NSF Grant DMS-1100714.}
\address{University of Tennessee, Knoxville, TN}
\author{Huy Tran}
\address{University of Washington, Seattle, WA}
\maketitle

\begin{abstract}
The Loewner equation encrypts a growing simple curve in the plane into a real-valued driving function. We show that if the driving function $\lbda$ is in $C^{\beta}$ with $\beta>2$ (or real analytic)  then the Loewner curve is in $C^{\beta + \frac{1}{2}}$ (resp.~analytic). This is a converse of \cite{EE} and extends the result in \cite{C}.
\end{abstract}  

\tableofcontents

\section{Introduction and results}

The Loewner differential equation, a classical tool that has attracted recent attention due to Schramm-Loewner evolution, provides a unique way of encoding a simple 2-dimensional curve  into a continuous 1-dimensional  function.  
In particular, let $\gamma : [0,T] \to \mathbb{C}$ be a simple curve with $\g(0) = 0$ and $\g(0,T) \in \mathbb{H} = \{ x+iy \, : \, y >0\}$.  For each $t\in [0,T]$,  there is a unique conformal map $g_t : \H \setminus \gamma(0,t) \to \H$ with the so-called hydrodynamic normalization: 
\begin{equation}\label{hydro norm}
g_t(z) = z + \frac{a(t)}{z} + O(z^{-2}) \text{ for $z$ near infinity.} 
\end{equation}
  Further, it is possible to reparametrize $\g$ so that $a(t) = 2t$ in equation \eqref{hydro norm}.  In this case, we say that $\g$ is parametrized by halfplane capacity 
  (since $a(t)$ is called the halfplane capacity of $\g[0,T]$ and can be thought of as a measure of the size of $\g[0,T]$.) 
  Unless stated otherwise, we will assume $\g$ has this parametrization throughout the paper.  The Loewner equation describes the time evolution of $g_t$:
  $$\partial_t g_t(z) = \frac{2}{g_t(z) - \lambda(t)},~~~g_0(z)=z,$$
 where $\lambda(t) = g_t(\g(t))$ is a continuous real-valued function, called the driving function. 
 (See  \cite{L} for further details.)

It is natural to ask how properties of the Loewner curve $\g$ correspond to properties of the driving function $\l$.  The results in this paper relate  the regularity of $\l$ to the regularity of $\g$.  Precise definitions of the regularity are given in Section 2.1, but at this point, we remind the reader that the Zygmund space $\Lambda_*^n$ is a generalization of $C^{n+1}$.

\begin{theorem}
\label{t: main theorem 1}
Let $\lbda \in C^\beta [0,T] $ for $\beta > 2$. 
Then the Loewner curve $\g$ is $C^{\beta+\frac{1}{2}} (0,T]$ 
      when $\beta+1/2 \notin \mb{N}$,
and $\g$ is in $\Lambda_*^{\beta - 1/2}(0, T]$ when $\beta+ 1/2 \in \mathbb{N}$.
\end{theorem}

See Theorem \ref{t: quantitative} for the quantitative version of this result.
This theorem  extends the work in \cite{C}, where the result was proven for $\beta \in (1/2, 2] \setminus \{3/2\}$. 
We do not know if the Zygmund space $\Lambda_*^n$ is optimal for the case $\beta=n+1/2$, but we do know that it is not possible to strengthen Theorem  \ref{t: main theorem 1} to say that $\g \in C^{n+1}$ when $\l \in C^{n+1/2}$.  This is illustrated in
 Section \ref{sec:ex}, in which we discuss an example where  $\l \in C^{3/2}$ but $\g$ fails to be $C^2$.

We also  address the analytic case:

\begin{theorem} \label{t:analytic}
If $\lbda$ is real analytic on $[0,T]$, then $\g$ is also real analytic on $(0,T]$. 
\end{theorem}

Notice that in both of these theorems, the regularity of $\g$ is on the  time interval $(0,T]$.
With the halfplane-capacity parametrization, it is not possible to extend these results to $t=0$.  
To see this, consider the  example when
the driving function is $\lambda(t) \equiv 0$.  Then the corresponding  Loewner curve is $\g(t) = 2i\sqrt{t}$.
Further, with the halfplane-capacity parametrization,  $\g(t)$ can always be expanded at $t=0$ in powers of $\sqrt{t}$, as we see in the following theorem.

\begin{theorem} \label{t:series at 0}
Assume that $\l \in C^{n+\alpha} [0,T]$ for $n \in \mathbb{N}$ and $\alpha \in (0, 1]$.  Then near  $t=0$, \begin{equation*}
\g(t) =  
\begin{cases}
 2i\sqrt{t} + a_2 t + i \,a_3 t^{3/2} + a_4 t^2 + \cdots + a_{2n} t^{n} + O(t^{n+\a})
	&\mbox{if   } \a \leq 1/2 \\ 
 2i\sqrt{t} + a_2 t + i \,a_3 t^{3/2} + a_4 t^2 + \cdots +  a_{2n} t^{n} +i \,a_{2n+1} t^{n+1/2}   + O(t^{n+\a})
 &\mbox{if   } \a > 1/2 \\ 
 \end{cases}
\end{equation*}
where the real-valued coefficients $a_m$ depend on 
$\l^{(k)}(0)$ for $k=1, \cdots, \lfloor \frac{m}{2} \rfloor$.
\end{theorem}

As Theorem \ref{t:series at 0} suggests, if we make the simple change of parametrization $t=s^2$, then the smoothness extends to $s=0$.

\begin{theorem}
\label{t: main theorem 2}
Let $\Gamma(s) = \g(s^2)$ be the reparametrized Loewner curve with driving function $\l$.
If $\l$ is real analytic on $[0,T]$, then $\G$ is real analytic on $[0,\sqrt{T}]$.
If $\lbda \in C^\beta [0,T] $, then $\G \in C^{\beta + 1/2}[0, \sqrt{T}]$ when $\beta + 1/2 \notin \mathbb{N}$.
 \end{theorem}

% FIGURE

\begin{figure}
\begin{tikzpicture}

\draw[thick] (0,0) to [out=145,in=220] (0.5,2) to [out=40,in=270] (1,2.8) 
to [out=90,in=300] (0.6,4);
\draw[fill] (0.5,2) circle [radius=0.05];
\node[right] at (0.5,2) {$\gamma(s)$};
\draw[fill] (1,2.8) circle [radius=0.04];	
\node[right] at (1,2.8) {$\gamma(s+u)$};
\draw (-3,0) -- (3,0);

\draw[->] (3,2) to [out=45,in=135] (5,2);
\node[above] at (4,2.5) {$g_s-\lambda(s)$};

\draw[fill] (0,0) circle [radius=0.05];
\node[below] at (0,0) {$0$};

\draw[fill] (8,0) circle [radius=0.05];
\node[below] at (8,0) {$0$};

\draw[thick] (8,0) to [out=70, in=290] (8.3,1.5) to [out=110, in=290] (7.8,2.5);
\draw[fill] (8.3,1.5) circle [radius=0.04];
\node[right] at (8.3,1.5) {$\gamma(s,s+u)$};
\draw (5,0) -- (11,0);

\end{tikzpicture}
\caption{The curve $\gamma(s,s+u) = g_s(\g(s + u)) - \lbda( s )$. 
}\label{gamma(s,s+u) figure}
\end{figure}

We wish to briefly describe the key tool used in this paper.
For $s\in [0,T]$, consider the simple curve $g_s(\g(s + u)) - \lbda( s )$,  which we denote by $\g( s, s + u), 0 \leq u\leq T - s $.  This is illustrated in Figure \ref{gamma(s,s+u) figure}. 
We are following the notation introduced in  \cite{C}, and to avoid confusion, we wish to point out that $\gamma(s, s+u)$ is {\it not} the image of the time interval $(s, s+u)$ under $\gamma$.
Rather, for fixed $s$ the curve $\g( s, s + u)$ corresponds to the time-shifted driving function $\lbda_s(u)= \lbda(u + s ) - \lbda(s)$, $0\leq u\leq T - s$. 
 It follows from \cite[Theorem 6.2]{C} that under the assumption $\lambda\in C^2 [0,T]$, the curve $\g$ is in $C^2$ and
 \begin{equation} \label{e: gamma''}
\g''(s)=\frac{2\g'(s)}{\g(s)^2} - 4\g'(s)\int^s_0 \frac{\partial_s [\g(s-u,s)]}{\g(s-u,s)^3}du.
\end{equation}
In order to understand the higher differentiability of $\g$, we need to understand $\g(s-u,s)$. Differentiating this function with respect to $u$, we obtain
\begin{equation} 
\label{e: g(s-u,s)}
\partial_u [\g(s-u,s)] = \partial_u [g_{s-u} ( \g(s)) - \lbda(s - u)] = \frac{-2}{\g(s - u, s)} + \lbda'(s - u) \mbox{ for } 0<u \leq s, 
\end{equation}
and $\g(s-u, s) |_{u=0} = \g(s,s) = 0$. We note that the above differential equation does not hold for $u=0$. This is the reason for us to investigate the following ODE:
\begin{eqnarray}
f'(u)&=&\frac{-2}{f(u)}+\lbda'(s-u),~~~~0\leq u\leq s, \label{e: ODE}\\
f(0) & = & i \e \in \mb{H}. \nonumber
\end{eqnarray}
The work in this paper depends on a deep understanding of the function $f(u) = f(u, s, \epsilon)$ which is the solution to \eqref{e: ODE}.  Once we show that $f(u, s, \e)$ converges uniformly to $\g(s-u, s)$ as $\e \to 0^+$ (see Lemma \ref{l: upward ODE}), we can use \eqref{e: gamma''} to translate information about $f$ into information about the derivatives of $\g$.

The paper is organized as follows:  
Section \ref{prelim section} includes initial properties of $f(u, s, \epsilon)$ 
    and some lemmas regarding solutions to a particular class of  ODEs.  
These lemmas will be useful in analyzing $f$ and its partial derivatives, and this is the content of Section \ref{Properties}.  
In Section \ref{smoothness section}, we state and prove a quantitative version of Theorem \ref{t: main theorem 1}.  
The real analyticity of the curve $\g$ in Theorem \ref{t:analytic} is proved in Section \ref{sec: analytic}.
In Section \ref{sec:behavior at 0}, we analyze the behavior of the trace at its base, proving Theorem \ref{t: main theorem 2} and Theorem \ref{t:series at 0}.  The latter is proven  by constructing
a nice curve that well-approximates a given Loewner curve at its base.
We conclude in Section \ref{sec:ex} with two examples.

\noindent {\bf Remark.}  Theorem \ref{t: main theorem 1} and Theorem \ref{t:analytic} provide a converse to the results of Earle and Epstein in \cite{EE}.  
Their results (translated from the radial setting to the chordal setting using \cite{M}) state that 
if any parametrization of $\g$ is $C^n$, then  the halfplane-capacity parametrization of $\g$ is in $C^{n-1}(0,T)$ and $\l \in C^{n-1}(0,T)$.  They also prove that if $\g$ is real analytic, then $\l$ must be real analytic.

{\bf Acknowledgement:} We appreciate the conversations and comments we received at various stages from Kyle Kinneberg, Michael Frazier, Donald Marshall, Steffen Rohde, Fredrik Johansson-Viklund and Carto Wong.  Part of this research was performed while the second author was visiting the Institute for Pure and Applied Mathematics (IPAM), which is supported by the National Science Foundation, and he thanks the institute for its hospitality and the use of its facilities. 
%
%---------------------------------------------------------------------------------%
\section{Preliminaries} \label{prelim section}
\subsection{Notation}

Let $I$ be an interval on the real line. The space $C^0(I)$ consists of all continuous functions on $I$ and $||\phi||_{\infty,I}=\sup_{t\in I} |\phi(t)|$ for $\phi\in C^0(I)$.

Let $\alpha\in (0,1)$. A function $\phi$ defined on $I$ is in $C^\alpha$ if $||\phi||_{\infty,I}<\infty$ and
its $\alpha-$H\"older norm is bounded:$$
||\phi||_{C^\alpha}:= \sup_{s,t\in I, s\neq t} \frac{|\phi(t) - \phi(s)|}{|t-s|^\alpha} <\infty.
$$

Let $n\in \mb{N}_0$, $\alpha\in [0,1]$ and $M>0$. A function $\phi$ is in $C^{n,\alpha}(I;M)$
 if $\phi',\cdots, \phi^{(n)}$ exist and are continuous and the following two conditions hold:
\begin{align*}
||\phi^{(k)}||_{\infty,I}\leq& \,M \, \mbox { for all } 0\leq k\leq n,\\
\text{ and }  ||\phi^{(n)}||_{C^\alpha}:=& \sup_{s,t\in I, s\neq t} \frac{|\phi^{(n)}(t) - \phi^{(n)}(s)|}{|t-s|^\alpha} \leq M.
\end{align*}
In particular,  the $n^{th}$ derivative of functions in $C^{n,1}$ are Lipschitz.  A function $\phi$ is in $C^n$ if $\phi\in C^{n,0}(I;M)$ for some $M$. When $\alpha\in(0,1)$, we also write $C^{n+\alpha}$ for $C^{n,\alpha}$.

Zygmund introduced a generalization of $C^{0,1}$ called $\Lambda_*$.  A continuous function $\phi$ is in $\Lambda_*(I)$ means that 
$$
||\phi||_{\Lambda_*}:= \sup_{s-\delta,s+\delta \in I, \delta > 0} 
					\frac{|\phi(s+\delta) +\phi(s-\delta)- 2\phi(s)|}{\delta} <\infty.
$$
We say that $\phi \in \Lambda^n_*(I;M)$ if  $\phi',\cdots, \phi^{(n)}$ exist and are continuous,  $\phi^{(n)} \in \Lambda_*$, and the following two conditions hold:
\begin{align*}
||\phi^{(k)}||_{\infty,I}\leq& \,M \, \mbox { for all } 0\leq k\leq n,\\
\text{ and }  ||\phi^{(n)}||_{\Lambda_*}\leq& M.
\end{align*}
Note that $ C^{n+1} \subset C^{n,1} \subset \Lambda_*^n $. 

The following proposition will be needed in Section \ref{sec:behavior at 0}.

\begin{proposition} \label{Cn_alpha}
If a function $\phi$ belongs to  $C^{n,\alpha}(I;M)$ then there exists $c=c(n, M)$ such that for all $t_0, t+ t_0\in I$,
$$|\phi(t + t_0) - \sum^n_{k = 0} \frac{1}{k!} t^k \phi^{(k)}( t_0)|\leq c t^{n + \alpha}.$$

\end{proposition}
\noindent
The proof follows from the integral form of the remainder of Taylor series. 

We use $C$ for a universal constant. For estimates related to a driving function $\lbda\in C^{n, \alpha}([0,T];M)$, we use $c$ for constants depending on $M,n,T$. When constants depend on other factors, we will state this explicitly. 

%---------------------------------------------------------------------------------%
\subsection{Loewner equation} \label{Loewner equation}

In the introduction we described how the Loewner equation can be used to encode a simple curve into its driving function.  This process can be reversed.
Let $\lbda$ be a real-valued continuous function on $[0,T]$ with $T>0$. Then the forward chordal Loewner equation is the following initial value problem:
\begin{equation} \label{forward LE}
\partial_t g_t (z)= \frac{2}{g_t(z)-\lbda(t)},~~~ g_0(z)=z.
\end{equation}
For each $z \in \mb{H}$, the solution $g_t(z)$ exists up to $T_z = \inf\{t>0: g_t(z)-\lbda(t)=0\}$. Let $K_t = \{ z\in \mb{H}: T_z\leq t\}$. It is known that $g_t$ is the unique conformal map from $\mb{H}\backslash K_t$ to $\mb{H}$ that satisfies
the hydrodynamic normalization at infinity:
$$g_t(z) = z + O(\frac{1}{z}),~~~~ \mbox{ near } z=\infty.$$
We say that $\lbda$ generates the curve $\g:[0,T] \to \overline{\mb{H}}$ if $\mb{H}\backslash K_t$ is the unbounded component of $\mb{H} \backslash \g(0,T]$. In this case $\lbda(t) = g_t(\g(t))$, $0\leq t\leq T.$ 
An important property of the chordal Loewner equation is the concatenation property, which says that for fixed $s$, the time-shifted driving function $\lambda(s + t)$ generates
 the mapped curve $g_{s}(\gamma(s, t))$. 
 For more details, see \cite{L}.

It was shown that if $\lbda\in C^{1/2} [0,T] $ with $||\lbda||_{C^{1/2}}<4$ then $\lbda$ generates a simple quasi-arc $\g$ (\cite{MR05}, \cite{Lind}). Since we work with  $\lbda\in C^\beta$ for $\beta>2$, on small intervals $||\lbda||_{C^{1/2}}\leq 1$. Therefore we are guaranteed that the corresponding Loewner curve is a simple curve. We can prove Theorems \ref{t: main theorem 1} and \ref{t:analytic} on small intervals, then use the concatenation property of the Loewner equation to derive the regularity of $\g$ on $[0,T]$. Henceforth, we assume $||\lbda||_{C^{1/2}}\leq 1.$

Changing \eqref{forward LE} by a negative sign gives the
backwards chordal Loewner equation:
\begin{equation} \label{backward LE}
\partial_t h_t (z)= \frac{-2}{h_t(z)-\xi(t)},~~~ h_0(z)=z
\end{equation}
for a continuous real-valued function $\xi$ defined on $[0,T]$.
The solution  $h_t(z)$ exists for all $z \in \H$ and $t \in [0, T]$, and $h_t$ is a conformal map from $\H$ into $\H$.  The forward and backward versions of the Loewner equation are related as follows:
if $g_t$ is the solution to \eqref{forward LE} with driving function $\l \in C[0, T]$ and $h_t$ is the solution to \eqref{backward LE} with driving function $\xi(t) = \l(T-t)$, then $h_t=g_{T-t} \circ g_T^{-1}$, and  in particular,  $h_T = g_T^{-1}$.

We think of \eqref{e: ODE} as a variant of the backward Loewner equation (with $\xi(u) = \l(s-u)$ and $f(u) = h_u(i \e) - \xi(u)$), and our first goal is to understand some basic properties of its solution
 $f(u) = f(u, s, \epsilon)$, when $(u,s) \in D :=  \{ (u,s) \, : \, 0 \leq u \leq s \leq T\}$.
 Further properties of $f(u,s,\epsilon)$ are in Section \ref{Properties}.

%
%%
%%%
%%%
%%%
%
%

\begin{lemma}
\label{l: upward ODE}
Let $\lbda \in C^1([0,T]; M)$, and let $0\leq s\leq T$ and $\e>0$.  Then  the ODE
\begin{eqnarray*}
f'(u)&=&\frac{-2}{f(u)}+\lbda'(s-u),~~~~0\leq u\leq s,\\
f(0) & = & i \e \in \mb{H}. \nonumber
\end{eqnarray*}
 has a unique solution $f(u)=f(u,s,\e)$, with $0\leq u\leq s$, satisfying the following properties:

\noindent 
(i) $\I f$ is increasing in $u$.

\noindent
(ii) For all $(u,s) \in D =  \{ (u,s) \, : \, 0 \leq u \leq s \leq T\}$
$$\sqrt{3u + \e^2} \leq \I f(u,s,\e)\leq \sqrt{4u+\e^2} $$
$$\mbox{and }\,\,\,\,\,\,|\R f(u,s,\e)| \leq \sqrt{u} \leq \frac{1}{\sqrt{3}} \I f(u,s,\e)  .$$

\noindent
(iii) For every $\delta>0$, there is $\e(\delta)>0$ such that
$$|f(u, s, \e_1) - f(u, s, \e_2)|\leq \delta\mbox{  for all } (u,s)\in D \text{ and } \e_1,\e_2 \leq \e(\delta).$$
In particular, $f(u, s, \e)$ converges uniformly as $\e\to 0+$ to a limit denoted by $f(u, s)$. This limit is the family of curves $\g(s - u, s)$ generated by $\lbda_s$, $0 \leq s \leq T$.  

(iv) Suppose $\lbda\in C^n([0,T];M)$, and let $l+k\leq n$ and $k\leq n-1$. Then $\partial^l_u \partial^k_s f$ exists and is continuous in $(u,s) \in D$ for all $\e>0$.

(v) If $\lbda\in C^n([0,T];M)$ and $1\leq k\leq n-1$, then
$\partial^k_s f(0,s,\e)=0 \mbox{ for all }  s\in [0,T]$ and  $\e>0.$

\end{lemma}
\begin{proof}
The equation (\ref{e: ODE}) is of the form:
\begin{eqnarray*}
f'(u)&=&G(f(u), u, s), 
\end{eqnarray*}
where $G(z, u, s) = \frac{-2}{z} + \lbda'(s - u)$ is jointly continuous in $z,u,s$, and Lipschitz in $z$ variable whenever $\I z\geq C>0$. So the solution exists on some interval containing 0. To show that the solution to (\ref{e: ODE}) exists on the whole interval $[0,s]$, it suffices to show that $(i)$ always holds. The idea of $(i)-(iii)$ comes from \cite{RTZ}, which contains a study of the Loewner equation when $||\lbda||_{C^{1/2}}<4$. For the convenience of the reader, we will present the proof here.

Let $x=x(u), y=y(u)$ be real and imaginary parts of $f(u)$. It follows from (\ref{e: ODE}) that
\begin{eqnarray}
(x + \lbda(s - \cdot))'  & = & \frac{-2x}{x^2 + y^2},\\
y' &= &\frac{2y}{x^2+ y^2}. \label{e: Dy}
\end{eqnarray}
In particular, $y$ is increasing and $(y^2)' \leq 4$. The former shows $(i)$, and the latter shows that $y\leq \sqrt{4u+\e^2}$.

\noindent
Now we will show that $|x(u)|\leq \sqrt{u}$, for $0 \leq u \leq s$.  
Suppose $0\leq x(u)$ and let $u_0 = \sup\{ v \in [0,u]: x(v)\leq 0\}$. So 
$$\partial_v (x(v) + \lbda(s - v)) \leq 0 \mbox{ for } u_0\leq v\leq u,$$
and 
$$ x(u) + \lbda( s- u ) \leq x(u_0) + \lbda( s- u_0) = \lbda( s- u_0).$$
Hence
$$x(u)\leq \lbda(s - u_0) - \lbda (s - u) \leq \sqrt{|u_0 - u|}\leq \sqrt{u}.$$
where the very last inequality follows since $||\lbda||_{1/2}\leq 1$. The same argument applies when $x(u)\leq 0$, proving that $|x(u)|\leq \sqrt{u}$.

\noindent
Next we will show $y(u) > \sqrt{3u}$ for $0\leq u\leq s$. Suppose this is not the case. 
Then since $y(0)=\e>0$, there exists $u_0\in (0,s]$ such that $y(u_0) = \sqrt{3u_0}$ and $y(u) \geq \sqrt{3u}$ for $u\in [0,u_0]$. It follows from (\ref{e: Dy}) that
$$(y^2)'=\frac{4y^2}{x^2 + y^2} \geq \frac{12u}{u + 3u} = 3 \mbox{ for } 0\leq u\leq u_0.$$
So $y(u_0)\geq \sqrt{3u_0 + \e^2} > \sqrt{3u_0}$. This is a contradiction. Therefore $y(u) > \sqrt{3u}$ and $(y^2)'\geq 3$. These show $(ii)$.

To show $(iii)$,  differentiate (\ref{e: ODE}) with respect to $\e$ to obtain
$$\partial_u(\partial_\e f) = \partial_\e \partial_u f = \frac{2 \partial_\e f}{f^2}.$$
Since $\partial_\e f(0, s, \e) = i$,
$$\partial_\e f(u, s, \e) = i \exp \int^u_0 \frac{2}{f^2(v, s, \e)}\, dv.$$
This implies
\begin{eqnarray*}
|\partial_\e f(u, s, \e)| &= &\exp \int^u_0 \R \frac{2}{f^2(v, s, \e)}\, dv \\
&= & \exp \int^u_0 \frac{2 (x^2(v) - y^2(v))}{(x^2(v) + y^2(v))^2}\, dv \leq 1.
\end{eqnarray*}
The last inequality comes from $(ii)$. It follows that 
$$|f(u,s,\e)-f(u,s,\e')|\leq |\e-\e'|, \mbox{ for all } 0\leq u\leq s\leq T,$$
and $f(u, s, \e)$ converges uniformly in $D$ to a limit, denoted by $f(u,s)$, as $\e\to 0^+$. 

Intuitively the limit $f(u,s)$ is equal to $\g(s-u,s)$ since $f(u, s, \e)$ satisfies the same ODE as $\g(s-u,s)$ does, and $\lim_{\e\to 0^+} f(0, s,\e) = \g(s - u, s) \large|_{u=0} = 0$. Indeed, from (\ref{e: g(s-u,s)}) and (\ref{e: ODE}) we can show that
\begin{equation} \label{e: difference}
| f(u,s, \e) - \g(s-u,s) | = | f(u_0, s, \e) - \g(s-u_0, s)| \exp\int^u_{u_0} \R \frac{2\,dv}{f(v, s, \e) \g(s-v, s) },
\end{equation}
with $0<u_0\leq u\leq s\leq T$ and $\e>0$.
Since $\g(s-v,s)$ is the tip of a Loewner curve generated by a driving function whose H\"older-1/2 norm is less than 1, then by \cite[Lemma 3.1]{C}, it satisfies
$$|\R \g(s - v,s)|\leq \I \g(s-v,s).$$
This implies that
 $$\R \frac{2}{f(v, s, \e) \g(s-v, s)} \leq 0.$$
Let $u_0\to 0^+$ and then $\e\to 0^+$ in (\ref{e: difference}) we get $f(u, s) = \g(s - u ,s)$.

Statement $(iv)$ follows from the standard ODE theory (see \cite{CL}, for instance) and the fact that $G$ is $C^{n-1}$ in $(u,s)$. 

We show $(v)$ by induction. For the base case,
 $$\partial_s f(0, s, \e)=\lim_{\delta \to 0}\frac{f(0, s + \delta, \e)-f(0, s,\e)}{\delta}=\lim_{\delta \to 0}\frac{\e - \e}{\delta}=0.$$
\noindent
Now supppose $\partial^k_s f(0,s,\e)=0$ for all $s\in [0,T]$. Then
$$\partial^{k+1}_s f(0, s, \e)=\lim_{\delta \to 0}\frac{\partial^k_s f(0, s + \delta, \e) - \partial^k_s f(0, s, \e)}{\delta}=0.$$
\end{proof}

\noindent
{\bf Remark.} For convenience, in this paper we only consider $\e\in (0,1]$. In this case, 
$$ \sqrt{3u} \leq |f(u, s, \e)| \leq \sqrt{Cu + \e^2} \leq C\sqrt{u} + C\e \leq c(T) \mbox{ for all } 0\leq u, s\leq T.$$ 

\noindent 
Later in Lemma \ref{l: partial^n_s f} we will show that $\partial^n_s f$ exists and is continuous in $(u,s)$.
%%

%------------------------------------------------------------------------%
\subsection{ODE lemmas}
The next lemma is frequently used in Section \ref{Properties} to investigate the regularity of $f(u, s, \e)$.
\begin{lemma} \label{l: ODE fact 1}
Consider a complex-valued function $X$ satisfying the initial value problem
$$X'(u)=P(u)X(u)+Q(u),~~~~ X(0)=0.$$
Suppose there exist constants $C, M_1 >0$ so that $|P(u)|\leq - C \R P(u)$ and $|Q(u)|\leq M_1$ for $0\leq u\leq u_0$. Then
$$|X(u)|\leq (C+1)M_1 u\mbox { for } 0\leq u\leq u_0.$$
\end{lemma}

\begin{proof} Solving the equation, one obtains
$$X(u) = R(u) + e^{-\mu(u)} \int^u_0 e^{\mu(v)}P(v)R(v)\, dv,$$
where $\mu (u)=-\int^u_0 P(v)\, dv$ and $R(u)=\int^u_0 Q(v) \, dv$. Since $|R(u)| \leq M_1 u$,
\begin{eqnarray*}
|X(u)| & \leq & M_1 u + M_1 u |e^{-\mu(u)}| \int^u_0 |e^{\mu(v)}| \cdot |P(v)| \, dv	\\
         & \leq & M_1 u + M_1 u e^{-\R \mu(u)} \int^u_0 e^{\int^v_0 -\R P(w) dw} C(-\R P(v)) \, dv	\\
         & = & M_1 u + C M_1 u e^{-\R\mu(u)} \left(e^{-\int^u_0 \R P(v) \, dv}-1\right)	\\
         & =   & M_1 u + C M_1 u e^{-\R \mu(u)} \left( e^{\R \mu(u)}-1\right) 	\\
         & \leq & (C+1) M_1 u.
\end{eqnarray*}
\end{proof}

%------------------------------------------------------------------------

In some cases, we will need a more general version of Lemma \ref{l: ODE fact 1}.

%------------------------------------------------------------------------------------------------------------------
\begin{lemma} \label{l: ODE fact 2}
Let $Y$ be a solution to
$$Y'(u)=P(u)Y(u) - P(u)Q(u) + R(u),~~~~Y(0)=Q(0)$$
with $|P|\leq -C\R P$ and $|Q(v)-Q(0)|\leq \w(v) $ on $[0,u_0]$, where $\w$ is a non-decreasing function and $C>0$.  

\noindent
(i) If $|R|\leq M_2u^{\beta-1}$ for some constant $M_2$, then
$$|Y(u) - Q(u)|\leq (C+1) \w(u) +(C+1)\frac{M_2}{\beta} u^\beta.$$
\noindent
(ii) If $Y(0)=Q(0)=0$ and $|R|\leq M_2$, then
$$|Y(u)|\leq C \w(u) + (C+1)M_2 u.$$
\noindent
(iii) More generally, 
$$|Y(u) - Q(u)|\leq (C+1) \w(u) +(C+1)\int^u_0 |R(v)| \, dv.$$
\end{lemma}
\begin{proof}
Let $\mu(u)=\int^u_0 -P(v) \, dv$ and $S(u)=\int^u_0 R(v) \, dv$. We have
\begin{align*}
Y(u) &=e^{-\mu(u)}Y(0) + e^{-\mu(u)}\int^u_0 e^{\mu(v)}(-P Q + R) \, dv \\
 &=Q(0)+e^{-\mu(u)}\int^u_0 e^{\mu(v)}(-P)[Q-Q(0)] \, dv + e^{-\mu(u)}\int^u_0 e^{\mu(v)}R \, dv \\
 &=Q(0)+e^{-\mu(u)}\int^u_0 e^{\mu(v)}(-P)[Q-Q(0)]\, dv+ S(u)-e^{-\mu(u)}\int^u_0 e^{\mu(v)}(-P) S \,dv,
\end{align*}
where the last equality follows from an integration by parts.
Therefore under the first assumption, $|S(u)| \leq M_2 u^\beta/\beta$ and
$$|Y(u)-Q(u)|\leq |Q(0)-Q(u)| + e^{-\R\mu(u)}\int^u_0 e^{\R \mu(v)}C (-\R P)\w(v) \, dv +|S(u)|$$
$$ +  e^{-\R \mu(u)}\int^u_0 e^{\R \mu(v)} C (-\R P) \frac{M_2}{\beta} u^\beta \,dv$$
$$\leq \w(u) + C \w(u) + \frac{M_2}{\beta} u^\beta +C\frac{M_2}{\beta} u^\beta.$$
Under the second assumption,
$$|Y(u)|\leq e^{-\R\mu(u)}\int^u_0 e^{\R \mu(v)}C (-\R P) \w(v)\, dv + |S(u)|  $$
$$ +  e^{-\R \mu(u)}\int^u_0 e^{\R \mu(v)} C(-\R P) M_2 u \,  dv$$
$$\leq C \w(u) +M_2 u + CM_2 u.$$
\end{proof}

%------------------------------------------------------------------------------------------------------
\section{Properties of $f(u,s,\e)$} \label{Properties}

In this section, we will prove all important properties of $f(u, s, \e)$, which are summarized in Proposition \ref{summary}. Then we will let $\e \to 0^+$ to obtain properties of $f(u, s) = \g(s - u, s)$. 
To accomplish this, we will show that $f(u,s, \e)$ and its partial derivatives satisfy the  type of ODE considered in Lemmas \ref{l: ODE fact 1} and \ref{l: ODE fact 2}.  These lemmas then provide  us with the needed estimates about $f(u,s, \e)$. The next two lemmas concern the $s$-derivatives of $f(u,s, \e)$.
%--------------------------------------------------------------------------------------
%
%

%----------------------------------------------------------------------------
\begin{lemma}\label{l: partial^k_s f} Suppose $\lbda\in C^n([0,T];M)$ with $n\geq 2$. For every $1\leq k\leq n-1$, there exists a function $ Q_k = Q_k(u,s,\e)$ such that
$$\partial_u (\partial_s^k f) = \frac{2}{f^2} \partial^k_s f + Q_k.$$
with $(u, s) \in D$, and $\e\in (0,1]$. Moreover there exists constant $c=c(M,n,T)>0$ so that
$$|\partial^k_s f(u, s, \e)|\leq c u.$$
\end{lemma}
\begin{proof} 
We will prove the lemma by induction. Let $k=1$ and $n\geq 2$. Fix $s\in [0,T]$ and $\e\in (0,1)$, and let $X(u)=\partial_s f(u,s, \e)$. Then
\begin{align*}
X'(u)=\partial_u\partial_s f(u, s, \e)=\partial_s \partial_u f(u, s, \e) &= \frac{2}{f^2(u, s, \e)}\partial_s f(u, s, \e)+\lbda''(s - u)\\
&=\frac{2}{f^2(u, s, \e)}X(u) + \lbda''(s - u),
\end{align*}
and $X(0)=\partial_s f(0, s, \e)=0$. Let $P_s=P_s(u, \e)=\frac{2}{f^2(u, s, \e)}$ and $Q_1(u, s, \e)=\lbda''(s - u)$. Clearly, $|Q_1|\leq M$. We will show that $P_s(\cdot, \e)$ satisfies the property of $P$ in Lemma \ref{l: ODE fact 1}.
Indeed, let $f(u, s, \e)=x+iy$. It follows from Lemma \ref{l: upward ODE}(ii) that there exists a constant $C>0$ such that 
$$|P_s(u, \e)|=\frac{2}{x^2+y^2}\leq -C\frac{2(x^2-y^2)}{(x^2+y^2)^2}=-C \,\R P_s(u, \e).$$
Applying Lemma \ref{l: ODE fact 1}, we obtain
$$|\partial_s f(u, s, \e)|\leq cu$$
completing the base case.

Now suppose  the lemma holds for $1\leq k-1 \leq n-2$ and $\partial_u (\partial^{k-1}_s f)=P_s\partial^{k-1}_s f + Q_{k-1}$. Then
$$\partial_u \partial^k_s f = \partial_s (\partial_u \partial^k_s f)=P_s\partial^k_s f + Q_k,$$
with $Q_k=\partial_s Q_{k-1}-\dfrac{4}{f^3}(\partial_s f)(\partial^{k-1}_s f).$
One can show by induction that 
$$Q_k=\lbda^{(k+1)}(s-u)+R_k$$
with $R_k(u, s, \e) =\sum \mbox{ terms}$, 
where the number of terms is no more than $k-1$ and each term has the form
$$\frac{c}{f^m}\prod^{m-1}_{j=1} \partial^{m_j}_s f,$$
for some $3\leq m\leq k+1$, and $1\leq m_j\leq k-1$. This term is by induction no bigger than
$$\frac{c}{u^{m/2}}u^{m-1}=cu^{m/2-1}\leq c(M,k,T) \sqrt{u}.$$
So $|Q_k|$ is bounded by a constant $c=  c(M,k,T)$, and hence Lemma \ref{l: ODE fact 1} implies that $|\partial^k_s f|\leq cu$. 
\end{proof}

\noindent
{\bf Remark.} $R_1=0$ and $R_k$ satisfies a recursive formula:
$$R_{k+1}(u, s, \e) = \partial_s R_k(u, s, \e) - \frac{4}{f(u, s, \e)^3} (\partial_s f(u, s, \e)) (\partial^k_s f(u, s, \e)).$$
We have shown that for $1\leq k\leq n-1$,
$$|R_k| \leq c(M, k, T) \sqrt{u}.$$
Since $R_n$ is only related to $\partial^k_s f$ for $0\leq k\leq n-1$, we have the same inequality:
$$|R_n| \leq c(M, n, T) \sqrt{u}.$$

%----------------------------------------------------------------------------------------------------
\begin{lemma}\label{l: partial^n_s f}
Suppose $\lbda\in C^n([0,T];M)$ then $\partial^n_s f(u, s, \e)$ exists and if $\lbda\in C^{n,\alpha}([0,T]; M)$ then 
$$|\partial^n_s f(u,s,\e)|\leq cu^\alpha,$$
where $c=c(M,n,T)$. 
\end{lemma}

\noindent 
{\bf Remark.} If $\lbda\in C^n([0,T]; M)$, then we can bound $ \displaystyle |\partial^n_s f(u, s, \e)|$ by the oscillation of $\lbda^{(n)}$:
$$ \displaystyle |\partial^n_s f(u, s, \e)|\leq c \sup\{|\lbda^{(n)}(u_1)-\lbda^{(n)}(u_2)|: |u_1-u_2|\leq u, u_1, u_2\in [0,s] \}\leq cM.$$

\begin{proof}
It follows from the proof of the previous lemma that
\begin{eqnarray*}
\partial_u (\partial^{n-1}_s f) &=&P_s \partial^{n-1}_s f + Q_{n-1}\\
&=& P_s \partial^{n-1}_s f + \lbda^{(n)}(s-u)+R_{n-1}.
\end{eqnarray*}
So
$$ \partial_u X = P_s X + Q, ~~~~X|_{ u = 0}=\lbda^{(n-1)}(s),
$$
where $X=\partial^{n-1}_s f + \lbda^{(n-1)}(s-u)$ and $Q=- P_s \lbda^{(n-1)}(s-u)+R_{n-1}$. Since $Q$ is $C^1$ jointly in $(u,s)$, $\partial_s X$ exists and satisfies
$$\partial_u (\partial_s X)=P_s \partial_s X - P_s \lbda^{(n)}(s-u) + R_n.
$$
and  $\partial_s X|_{u=0}=\lbda^{(n)}(s)$. Hence $\partial^n_s f$ exists and is continuous in $(u,s)$. Since $|R_n|\leq c(M,n,T)$, apply Lemma \ref{l: ODE fact 2} $(i)$ with $\w \equiv Mu^\alpha$, $M_2=c$, and $\beta=1$ to obtain
$$|\partial^n_s f|=|\partial_s X-\lbda^{(n)}(s-u)|\leq (C+1) Mu^\alpha  + cu\leq cu^\alpha.$$
\end{proof}

The next three lemmas concern the oscillation of $\partial^k_s f$ in the variable $s$.  
In the proofs, we omit $\e$ from the formulas at times (for ease of reading), but we remind the reader that the functions $f, P_s, Q_k, R_k$ do depend on the three variables $u,s,\e$.
%-----------------------------------------------------------------------------
\begin{lemma} \label{l: lambda in C^1}
Suppose $\lbda\in C^{1,\alpha}([0,T];M)$ with $\alpha \in (0,1]$. Then
$$|f(u, s + \delta, \e)-f(u, s, \e)|\leq  c\min(u \delta^\alpha,\delta u^\alpha),$$
$$|\partial_s f(u, s + \delta, \e)-\partial_s f(u, s, \e)|\leq c(1 + \frac{\e}{\alpha}) \min (u^\alpha, \delta^\alpha)$$
for $0\leq u\leq  s\leq s+\delta\leq T$ and $\e>0$. 
\end{lemma}
\begin{proof}
Since $|\partial_s f(u, s, \e)|\leq c u^\alpha$ (by Lemma \ref{l: partial^n_s f}),
$$|f(u, s + \delta, \e) - f(u, s, \e)| \leq c \delta u^\alpha.$$
Omitting the parameter $\e$ for convenience, we have
$$\partial_u [f(u, s + \delta) - f(u, s)]=\frac{2}{f(u, s) f(u, s + \delta)}[f(u, s + \delta) - f(u, s)] + \lbda'(s + \delta - u) - \lbda'(s - u),$$
and $f(0, s + \delta) -  f(0, s)=0$. We see that $P:=\dfrac{2}{f(u, s) f(u, s + \delta)}$ satisfies
$$|P(u)|\leq - C \R P(u)$$
and that $Q=\lbda'( s + \delta - u) - \lbda'( s - u)$ is bounded by $M \delta^\alpha$.
Therefore, Lemma \ref{l: ODE fact 1} implies
$$|f (u, s + \delta) - f(u, s)|\leq CM u \delta^\alpha.$$

It remains to prove the last inequality. We have
$$\partial_u [\partial_s f(u, s + \delta) + \lbda'(s + \delta - u)] = P_{s+\delta} \partial_s f(u, s + \delta),$$
and
$$\partial_u [\partial_s f(u, s) + \lbda'(s - u)] = P_{s} \partial_s f(u, s).$$
So
\begin{align*}
\partial_u [\partial_s f(u, s + \delta) + \lbda'&(s + \delta - u) - \partial_s f(u, s) - \lbda'(s - u)] \\
= P_{s + \delta} & [\partial_s f(u, s + \delta) +\lbda'(s + \delta - u) - \partial_s f(u, s) - \lbda'(s - u)] \\
 &- P_{s + \delta} \left(\lbda'(s + \delta - u)  - \lbda'(s - u)\right) + (P_{s+\delta}-P_s)\partial_s f(u, s).
 \end{align*}
We will apply Lemma \ref{l: ODE fact 2} with $Q(u) = \lbda'(s + \delta - u)  - \lbda'(s - u)$ and $R(u) = (P_{s+\delta}-P_s)\partial_s f(u, s)$.
Note
$$|\lbda'(s+\delta-u)-\lbda'(s-u)-\lbda'(s+\delta)+\lbda'(s)|\leq 2 M \min (u^\alpha, \delta^\alpha).$$
Further 
\begin{align*}
|P_{s+\delta}-P_s|\cdot |\partial_s f(u, s)| &\leq \frac{c |f(u, s + \delta) - f(u, s)| \cdot |f(u, s) + f(u, s + \delta)|}{u^2}u^\alpha\\
 &\leq \frac{ c u \delta^\alpha \sqrt{Cu+ \e^2}}{u^2}u^\alpha\\
&\leq c \delta^\alpha u^{\alpha - 1/2} + c \delta^\alpha\e u^{\alpha - 1},
\end{align*}
and so
$$\int^u_0 |R(v)| \, dv \leq \int^u_0  \left( c \delta^\alpha v^{\alpha - 1/2} + c \delta^\alpha\e v^{\alpha - 1} \right) \, dv \leq  c\delta^\alpha u^{\alpha + 1/2} + c \delta^\alpha \frac{\e}{\alpha} u^\alpha.$$
Therefore, by Lemma \ref{l: ODE fact 2} $(iii)$ with $\w \equiv 2M \min (u^\alpha, \delta^\alpha)$, 
\begin{align*}
 |\partial_s f(u, s + \delta) - \partial_s f(u, s)| &\leq C M \min (u^\alpha, \delta^\alpha) + c\delta^\alpha u^{\alpha + 1/2} + c \delta^\alpha \frac{\e}{\alpha} u^\alpha\\
&\leq c (1+ \frac{\e}{\alpha}) \min (u^\alpha, \delta^\alpha).
\end{align*}
\end{proof}
%---------------------------------------------------------------------------
\begin{lemma}
Suppose $\lbda\in C^{n,\alpha}([0,T];M)$ with $n\geq 2$ and $\alpha \in (0,1]$. Then
$$|R_k(u, s + \delta, \e) - R_k (u, s, \e)| \leq c \delta \sqrt{u}~~~~ \mbox{ when } ~~~~ 1\leq k\leq n-1,$$
and 
$$|\partial^k_s f(u, s + \delta,\e) - \partial^k_s f(u, s, \e)|\leq cu\delta ~~~~ \mbox{ when }~~~~ 1\leq k\leq n-2, $$ 
and 
$$|\partial^{n-1}_s f(u, s + \delta,\e) - \partial^{n-1}_s f(u, s, \e)|\leq c \min( u^\alpha \delta, u \delta^\alpha).$$
\end{lemma}
\begin{proof}
From the Remark following Lemma  \ref{l: partial^k_s f}, we know that 
 $R_1=0$,  $R_k$ satisfies the recursive formula:
$$R_{k+1} = \partial_s R_k - \frac{4}{f^3}(\partial_s f)(\partial^k_s f),$$
and  $|R_k| \leq c \sqrt{u}$ for $1 \leq k \leq n$.
Therefore, for $k+1\leq n$, Lemma \ref{l: partial^k_s f} implies that
\begin{align*}
|\partial_s R_k | &\leq |R_{k+1} |+\frac{4}{|f|^3}|\partial_s f|\cdot | \partial^k_s f|\\
&\leq c \sqrt{u}.
\end{align*}
Thus
$$|R_k(u, s + \delta, \e)-R_k(u, s, \e)|\leq \int^{s+\delta}_s |\partial_sR_k (u, r, \e)|dr \leq c \delta \sqrt{u},$$
proving the first statement.  

When $1 \leq k \leq n-2$, Lemma \ref{l: partial^k_s f} implies that
$$|\partial^k_s f(u,s+\delta,\e)-\partial^k_s f(u,s,\e)|\leq \int^{s + \delta}_s |\partial^{k+1}_s f(u, r, \e)| dr \leq c u \delta,$$
proving the second statement.
From Lemma \ref{l: partial^n_s f}
$$|\partial^{n-1}_s f(u, s + \delta,\e) - \partial^{n-1}_s f(u, s, \e)|\leq \int^{s + \delta}_s |\partial^n_s f(u, r, \e)| dr \leq c u^\alpha \delta.$$

To prove the third statement, it remains to show
\begin{equation} \label{third statement goal}
|\partial^{n-1}_s f(u, s + \delta, \epsilon) - \partial^{n-1}_s f(u, s, \epsilon)|\leq c \delta^\alpha u.
\end{equation}
Omitting the parameter $\e$, we have
\begin{align*}
\partial_u [\partial^{n-1}_s f(u, s + \delta) - \partial^{n-1}_s f(u, s)]  = P_{s + \delta} & [\partial^{n-1}_s f(u, s + \delta) - \partial^{n - 1}_s f(u, s)] \\
 &+ (\lbda^{(n)}(s + \delta - u)  - \lbda^{(n)}(s - u))\\
 &+ (P_{s + \delta} - P_s)\partial^{n-1}_s f(u, s) \\ &+ R_{n - 1}(u, s + \delta) - R_{n - 1}(u, s).
 \end{align*}
Since 
$$|\lbda^{(n)}(s + \delta - u)  - \lbda^{(n)}(s - u)|\leq M \delta^\alpha,$$
and
$$|P_{s+\delta} - P_s|\cdot |\partial^{n-1}_s f(u, s)|\leq \frac{c\delta u C}{u^2} u \leq c \delta \leq c \delta^\alpha, $$
and 
$$|R_{n - 1}(u, s + \delta) - R_{n - 1}(u, s)|\leq c\delta \sqrt{u} \leq c \delta^\alpha,$$
we apply Lemma \ref{l: ODE fact 1} with $M_1=c \delta^\alpha$ to prove \eqref{third statement goal}.
 \end{proof}
%----------------------------------------------------------------------------
\begin{lemma}
Suppose $\lbda\in C^{n,\alpha}([0,T];M)$ with $n\geq 2$ and $\alpha\in (0,1]$. There exists $c=c(M,n,T)$ so that
\begin{eqnarray*}
|R_{n+1}(u, s, \e)| & \leq & cu^{\alpha-1/2}, \\
|R_n(u, s + \delta, \e) - R_n(u, s, \e)|  & \leq & cu^{\alpha-1/2}\delta,\\
|\partial^n_s f(u,s+\delta,\e) - \partial^n_s f(u,s,\e)| & \leq & c (1+ \frac{\e}{\alpha}) \min (u^\alpha, \delta^\alpha).
\end{eqnarray*}
\begin{proof}
Let's note that
$$R_n = \sum \frac{c}{f^m}\prod^{m-1}_{j=1} \partial^{m_j}_s f$$
with $3\leq m\leq n+1$, $1\leq m_j\leq n-1$, and the number of terms in the sum is no more than $n-1$. Since $\partial^n_s f$ exists, so does $R_{n+1}$:
$$R_{n+1}=\sum \frac{c}{f^m}\prod^{m-1}_{j=1} \partial^{m_j}_s f,$$
with $3\leq m\leq n+2$ and $1\leq m_j\leq n$. We can check that in each product, there is at most one $m_j=n$. Hence
$$|R_{n+1}|\leq cn\frac{u^{m-2}u^\alpha}{u^{m/2}}\leq c u^{\alpha+m/2-2}\leq c(M,n,T) u^{\alpha-1/2},$$
and
$$|\partial_s R_n|\leq |R_{n+1}|+\frac{4}{|f|^3}|\partial_s f|\cdot |\partial^n_s f| \leq cu^{\alpha-1/2}.$$
This implies that
$$|R_n(u,s+\delta)-R_n(u,s)|\leq cu^{\alpha-1/2}\delta.$$

It remains to prove the last statement.
Now we have
$$\partial_u(\partial^n_s f(u, s + \delta) + \lbda^{(n)}(s + \delta - u)) = P_{s + \delta}\partial^n_s f(u, s + \delta) + R_n(u, s + \delta),$$
and
$$\partial_u(\partial^n_s f(u, s)+\lbda^{(n)}(s - u)) = P_s \partial^n_s f(u, s) + R_n(u, s).$$
Let 
\begin{align*}
Y(u) &= \partial^n_s f(u, s + \delta) + \lbda^{(n)}(s + \delta - u)-\partial^n_s f(u, s) - \lbda^{(n)}(s - u)  \text{ and } \\
 Q(u) &=\lbda^{(n)}(s + \delta - u) -\lbda^{(n)}( s - u).
\end{align*}
Then
$$\partial_u Y = P_{s+\delta} Y - P_{s+\delta} Q + (P_{s + \delta} - P_s) \partial^n_s f(u, s) + R_n(u, s + \delta) - R_n (u, s).$$
We see that
$$|Q(u)-Q(0)|\leq c\min (u^\alpha, \delta^\alpha),$$
and
$$|(P_{s + \delta} - P_s) \partial^n_s f(u, s)|\leq \frac{cu\delta \sqrt{Cu+\e^2}}{u^2} u^\alpha \leq c \delta u^{\alpha-1/2} + c\e \delta u^{\alpha - 1}.$$
By Lemma \ref{l: ODE fact 2} $(iii)$ with $|R(u)| \leq c \delta u^{\alpha - 1/2} + c \e \delta u^{\alpha - 1}$,  
$$|\partial^n_s f(u, s + \delta, \e) - \partial^n_s f(u, s, \e)| = |Y-Q| \leq c\min (u^\alpha, \delta^\alpha) +  c \delta u^{\alpha + 1/2} + \frac{c\e \delta}{\alpha} u^\alpha. $$
\end{proof}
\end{lemma}

%-----------------------------------------------------------------------------------------------------
\begin{lemma} \label{l: mix_u_s}
 (Boundedness of mixed $u$ and $s$ derivatives.) 
 Suppose $\lbda\in C^n([0,T];M)$. Let $s_0\in (0,T)$ and $D_0=\{ (u,s)\in D: s_0\leq u\}$. 
 There exists $L_0=L_0(M,n,T,s_0)$ such that for all $l+k\leq n$, 
$$|\partial^l_u \partial^k_s f( u, s, \e)| \leq L_0.$$ 
In other words, $f \in C^n(D_0; L_0)$ for every $\e\in (0,1].$

\end{lemma}
\begin{proof}
	The case  $l=0$ and $k\leq n$ is proven by Lemmas \ref{l: partial^k_s f} and \ref{l: partial^n_s f}. Consider $k=0$ and $1\leq l\leq n$. We have 
$$\partial_u f = \frac{-2}{f} + \lbda'(s-u).$$
This implies that when $u_0\leq u$,
$$|\partial_u f|\leq \frac{2}{C\sqrt{u}} + M \leq L_0.$$
We can show by induction in $l$ that
$$\partial^l_u f = \frac{2}{f^2} \partial^{l-1}_u f +(-1)^{l-1} \lbda^{(l)}(s-u) + \hat{R}_l,$$
where $\hat{R}_l$ is the sum of a finite number (depending on $l$) of terms of the form
$$\frac{c}{f^m} \prod^{m-1}_{j=1} \partial^{m_j}_u f$$
with $3\leq m\leq l-1$ and $1\leq m_j\leq l-2$. Hence by induction $|\partial^l_u f|\leq L_0$ for $s_0\leq u\leq T$. The other cases $1\leq k\leq n-1$ are proved similarly.

\end{proof}

In summary, we have proved the following results about $f(u, s, \e)$:

\begin{proposition} \label{summary}
If $\l$ is in $C^{n, \alpha}[0,T]$, then
$f(u, s, \e)$ satisfies the following properties:
\begin{itemize}
\item $C\sqrt{u+\e^2}\leq |f(u, s, \e)|\leq C'\sqrt{u} + C'\e$.  
\smallskip
 \item $|\partial^k_s f(u, s, \e)|\leq cu$ for $1\leq k\leq n-1$.
\smallskip
\item $|\partial^n_s f(u, s, \e)|\leq cu^\alpha$.
\smallskip
\item $|\partial^k_s f(u, s + \delta, \e) - \partial^k_s f(u, s, \e)|\leq c u \delta$ for $1 \leq k\leq n-2$.
\smallskip
\item $|\partial^{n-1}_s f(u, s + \delta, \e) - \partial^{n-1}_s f(u, s, \e)|\leq c\min (u\delta^\alpha, u^\alpha \delta)$ if $0\leq n-1$.
\smallskip
\item $|\partial^n_s f(u, s + \delta, \e) - \partial^n_s f(u, s, \e)| \leq c (1 + \frac{\e}{\alpha}) \min (u^\alpha, \delta^\alpha)$ for $1\leq n$. 
\smallskip
\item For every $0<s_0<T$, there exists $L_0=L_0(M, n, T, s_0)$ such that for all $l+k\leq n$, $|\partial^l_u \partial^k_s f( u, s, \e)| \leq L_0$.
\end{itemize}
\end{proposition}

We emphasize that $c$ depends only on $M,n,T$, not on $\alpha$ and $\e$. We know from Lemma \ref{l: upward ODE} that $f(u,s,\e)$ converges uniformly in $D$ to $f(u,s)$ as $\e \to 0^+$. For all $l+k=n$, it follows from the proof of previous lemmas that $\partial^l_u \partial^k_s f(u,s,\e)$ can be expressed in terms of lower derivatives in $u$ and $s$ of $f(u,s,\e)$. Therefore in $D_0=\{(u,s)\in D: 0<s_0\leq u\leq s\leq T\}$, $\partial^l_u \partial^k_s f(u,s,\e)$ converges uniformly. This implies the following:

\begin{corollary} \label{sumry-prop}
If $\l$ is in $C^{n, \alpha}[0,T]$, then
 $f(u,s)$ is in $C^n (D_0)$ and satisfies 

\begin{itemize}
\item $C\sqrt{u}\leq |f(u, s)|\leq C'\sqrt{u}$.  
\smallskip
\item $|\partial^k_s f(u, s)|\leq cu$ for $1\leq k\leq n-1$.
\smallskip
\item $|\partial^n_s f(u, s)|\leq cu^\alpha$.
\smallskip
\item $|\partial^k_s f(u, s + \delta) - \partial^k_s f(u, s)|\leq c u \delta$ for $1 \leq k\leq n-2$.
\smallskip
\item $|\partial^{n-1}_s f(u, s + \delta) - \partial^{n-1}_s f(u, s)|\leq c\min (u\delta^\alpha, u^\alpha \delta)$ if $0\leq n-1$.
\smallskip
\item $|\partial^n_s f(u, s + \delta) - \partial^n_s f(u, s)| \leq c \min (u^\alpha, \delta^\alpha)$ for $1\leq n$. 
\smallskip
\item For every $0<s_0<T$, there exists $L_0=L_0(M, n, T, s_0)$ such that for all $l+k\leq n$, $|\partial^l_u \partial^k_s f( u, s)| \leq L_0$.
\smallskip
\end{itemize}
\end{corollary}
The first three properties of the corollary will help to show that we can take derivatives of the integral term in the formula (\ref{e: gamma''}). The next three properties will be used to estimate the H\"older norm of the derivatives.
%--------------------------------------------------------------------------------------------------------------
\begin{corollary} If $\lbda$ is in  $C^{n,\alpha}  [0,T] $ with $n\geq 2$ and $\alpha \in (0,1]$, then $\g$ is in $C^n (0,T] $.
\end{corollary}
\begin{proof}  The previous arguments imply that $\g(s-u,s)\in C^n(D_0)$ for every $s_0\in (0,T)$. Hence $s\mapsto \g(0,s)\in C^n (0,T] $. Since $\g(s) = \g(0,s) + \lbda(0)$, the curve $\g$ is in $C^n (0,T]$. 
\end{proof}

%-----------------------------------------------------------------------------------------------------------
\section{Smoothness of $\g$}\label{smoothness section}
The goal of this section is to prove the following quantitative version of Theorem \ref{t: main theorem 1}.

\begin{theorem}
\label{t: quantitative}
Suppose $\lbda\in C^{n,\alpha}([0,T];M)$ with $n\geq 2$ and $\alpha \in (0,1]$.

(i) If $\alpha < 1/2$, then $\gamma \in C^{n, \alpha + 1/2} (0, T] $. For every $0 < s_0 < T$, there exists $c_0 = c_0 (M,n,T,s_0)$ such that $\g\in C^n( [s_0,T];c_0)$ and
$$|\g^{(n)}(s + \delta) - \g^{(n)}(s)| \leq \frac{c_0}{1 - 2 \alpha} \delta^{\alpha + 1/2},$$

(ii) If $\alpha = 1/2$, then $\gamma \in \Lambda_*^n(0,T] $. For every $0< s_0 <T$, there exists $c_0=c_0 (M,n,T,s_0)$ such that $\g\in C^n([s_0,T];c_0)$ and
$$|\g^{(n)}(s + \delta) + \g^{(n)}(s - \delta)- 2\g^{(n)}(s)| \leq c_0 \delta.$$

(iii) If $\alpha\in (\frac{1}{2},1]$, then $\g\in C^{n+1,\alpha - 1/2}(0,T]$. For every $0< s_0 <T$,  there exists $c_0=c_0(M, n, T, s_0)$ such that $\g\in C^{n+1}([s_0,T];c_0)$ and
 $$|\g^{(n + 1)}(s + \delta) - \g^{(n + 1)}(s)| \leq \frac{c_0}{2\alpha - 1} \delta^{\alpha - 1/2}.$$
\end{theorem}

\begin{proof}
Assume that $\lbda\in C^{n,\alpha}([0,T];M)$ with $n \geq 2$ and $\alpha \in (0,1]$.
Fix $s_0\in (0,T)$ and let $D_0 = \{(u,s)\in D: 0<s_0\leq u\leq s \leq T\}$.  Recall from \cite{C} that
$$\g''(s)=\frac{2\g'(s)}{\g(s)^2} - 4\g'(s)\int^s_0 \frac{\partial_s [f(u,s)]}{f(u,s)^3}\,du.$$
We need to show
$$ F(s) :=  \int^s_0 \frac{\partial_s f(u,s)}{f(u,s)^3}\, du \mbox{ is } \left\{
 \begin{array}{rcl} 
\mbox{ in } C^{n-2} & \mbox{ and } & F^{(n-2)} \in C^{\alpha + 1/2} \mbox{ when } \alpha \in (0, 1/2) \\
\mbox{ in } C^{n - 2} & \mbox{ and } & F^{(n - 2)} \in \Lambda_*  \mbox{ when } \alpha = 1/2\\
\mbox{ in } C^{n - 1} & \mbox{ and } & F^{(n - 1)} \in C^{\alpha - 1/2} \mbox{ when } \alpha \in (1/2,1]
\end{array}\right. .$$
Let $F_1(u,s)= \dfrac{\partial_s f(u,s)}{f(u,s)^3}$ and $\hat{R}_1 (u,s)  = 0$. We define $F_k$ and $\hat{R}_k$ recursively as follows:
\begin{eqnarray*}
\hat{R}_k &=& \partial_s \hat{R}_{k - 1} - \frac{3 (\partial_s f) (\partial^{k - 1} _s f)}{f^4},\\ 
F_k  & = &\partial_s F_{k - 1} = \frac{\partial^k_s f }{ f^3} + \hat{R}_k.
\end{eqnarray*}
Let  $\hat{F}_k(s)=F_k(s, s)$. Then formally
\begin{equation}\label{e: F^{(n-2)}}
 F^{(n-2)}(s) = \hat{F}_1^{(n-3)}(s)+ \hat{F}_2^{(n-4)}(s) + \cdots \hat{F}_{n-2}(s) + \int^s_0 \left[ \frac{\partial^{n-1}_s f(u, s)}{f^3(u, s)} + \hat{R}_{n - 1}(u, s)\right]\, du,
\end{equation}
and
\begin{equation}\label{e: F^{(n-1)}}
 F^{(n-1)} (s)= \hat{F}_1^{(n-2)}(s)+ \hat{F}_2^{(n-3)}(s) + \cdots \hat{F}_{n-1}(s) + \int^s_0 \left[ \frac{\partial^n_s f(u, s)}{f^3(u, s)} + \hat{R}_n(u, s)\right] \, du.
\end{equation}

We notice that 
\begin{equation} \label{Rkhat}
\hat{R}_k = \sum\frac{c}{f^m}\prod^{m-2}_{j=1} (\partial^{m_j}_s f),
\end{equation}
where there are at most $k-1$ terms for the sum, $4\leq m\leq k+2$, and $1\leq m_j\leq k-1$.  
Further, when $k \geq 3$  each product contains at most one $m_j=k-1$. 
Therefore, $\hat{R}_k \in C^{n-(k-1)}(D_0)$, $F_k \in C^{n-k}(D_0)$ and $\hat{F}_k\in C^{n-k} [s_0,T] $.  The representation of $\hat{R}_k$ in \eqref{Rkhat} also implies that
\begin{eqnarray}
\label{e: R^k} |\hat{R}_k(u, s)| &\leq & c \mbox{ for } 1 \leq k \leq n,\\
\label{e: R^{n+1}} \mbox{ and }~~|\hat{R}_{n+1}(u, s)| &\leq & \frac{c}{u^{1/2}} \mbox{ if } \alpha \geq \frac{1}{2}.
\end{eqnarray}
Hence equation (\ref{e: F^{(n-2)}}) holds for all $\alpha \in (0,1]$ and equation (\ref{e: F^{(n-1)}}) holds when $\alpha \in (1/2,1]$. 

Let 
$$\displaystyle I_k(s):=\int^s_0 \dfrac{\partial^k_s f(u,s)}{f(u,s)^3}\, du \; \text{  and  } 
\; \displaystyle IR_k(s)=\int^s_0 \hat{R}_k (u, s)\, du.$$
Theorem \ref{t: quantitative} will be proven once we show that
\begin{itemize}
\item $I_{n - 1} + IR_{n - 1} \in C^{\alpha + 1/2} [s_0,T] \; \text{  for  } \alpha \in (0,1/2), $ \smallskip
\item $I_{n - 1} + IR_{n - 1} \in  \Lambda_* [s_0,T] \; \text{  for  } \alpha = 1/2, \; \text{  and  }$  \smallskip
\item $I_n + IR_n \in C^{\alpha - 1/2}  [s_0,T] \; \text{  for  } \alpha \in (1/2,1],$
\end{itemize}
along with the needed bounds on $| I_k(s+\delta) - I_k(s) |$ and $| IR_k(s + \delta) - IR_k(s)|$ (and the appropriate estimates for the $\alpha = 1/2$ case.)
This is the content of the next three lemmas.

\end{proof}

\begin{lemma}
Suppose $\lbda\in C^{n,\alpha}([0,T]; M)$, with $n\geq 2$ and $\alpha \in (0,1]$.
Then  there exists $c=c(M,n,T)$ such that for all $0<s_0\leq  s\leq s+ \delta \leq T$,
\begin{eqnarray*} 
|IR_k(s + \delta) - IR_k(s)|  & \leq & c \delta \mbox{ for all } 1\leq k \leq n-1 \text{  and  }\\
|IR_n(s + \delta) - IR_n(s) |  & \leq & c \delta \mbox{ if } \alpha \geq \frac{1}{2}.
\end{eqnarray*}
\end{lemma}

\begin{proof}
It follows from the definition of $\hat{R}_k$ and formula (\ref{e: R^k}) that 
for $1\leq k\leq  n - 1$,
$$
|\hat{R}_k(u, s + \delta) - \hat{R}_k(u, s)|   \leq \int_s^{s + \delta} | \partial_v \hat{R}_k(u,v) |\, dv
\leq   c\delta.$$
Similarly if $\alpha \geq \frac{1}{2}$ equation \eqref{e: R^{n+1}} implies
$$|\hat{R}_n(u, s + \delta) - \hat{R}_n(u, s)|   \leq  \frac{c\delta}{u^{1/2}}.$$
Integrating completes the lemma.
\end{proof}

\begin{lemma}Suppose $\lbda\in C^{n,\alpha}([0,T]; M)$, with $n\geq 2$ and $\alpha \in (0,\frac{1}{2}]$.
Then $I_{n-1} \in C^{\alpha + 1/2} [s_0, T] $ when $ \alpha \in (0,1/2)$ 
 and $I_{n-1} \in \Lambda_*[s_0,T] $ when $ \alpha = 1/2.$
In particular,
 there exists $c=c(M,n,T)$ such that for all $0<s_0\leq  s\leq s+ \delta \leq T$,
$$|I_{n - 1}(s + \delta) - I_{n - 1} (s) | \leq \left\{\begin{array}{rcl} c(\frac{1}{1 - 2 \alpha} +1) \delta^{\alpha + 1/2} + c(1+\frac{1}{\sqrt{s_0}}) \delta & \mbox{ when} & 0< \alpha < \frac{1}{2}\\
c (1+\log^+ \frac{s}{\delta} + \frac{1}{\sqrt{s_0}})\delta& \mbox{ when } & \alpha = \frac{1}{2} \end{array}\right. $$
and when $\alpha = 1/2$, 
\begin{equation}\label{LambdaStarEstimate}
|I_{n - 1}(s + \delta) + I_{n - 1}(s - \delta)- 2 I_{n - 1} (s) | \leq c\left(1+\frac{1}{\sqrt{s_0}}\right) \delta
\end{equation}
for all $0<s_0\leq  s-\delta \leq s+ \delta \leq T$.
\end{lemma}

\begin{proof}
We decompose $I_{n-1}(s+\delta) - I_{n-1}(s)$ into the sum of four integrals and bound each integral.
\begin{eqnarray*}
 I_{n-1}(s + \delta) - I_{n-1} ( s) &= &\int^{\delta \wedge s}_0 \frac{ \partial^{n-1}_s f(u, s + \delta) - \partial^{n-1}_s f(u, s)}{f(u, s + \delta)^3}\,du\\
& +           &     \int^s_{\delta \wedge s} \frac{ \partial^{n-1}_s f(u, s + \delta) - \partial^{n-1}_s f(u, s)}{f(u, s + \delta)^3}\,du\\
&      +     &  \int^s_0 \frac{\partial^{n-1}_s f(u,s) (f(u,s)^3 - f(u, s + \delta)^3)}{f(u, s)^3 f(u, s + \delta)^3}\, du\\
&     +      & \int^{s + \delta}_ s \frac{\partial^{n-1}_s f(u, s + \delta)}{f(u, s + \delta)^3}\,du.
\end{eqnarray*}
The first integral:
\begin{eqnarray*}
\left |\int^{\delta \wedge s}_0 \frac{ \partial^{n-1}_s f(u, s + \delta) - \partial^{n-1}_s f(u, s)}{f(u, s + \delta)^3}\,du\right| &\leq & \int^{\delta \wedge s}_0 \frac{c u \delta^\alpha}{u^{3/2}}\, du  \\
&=& c \delta^\alpha \sqrt{\delta \wedge s} \leq c \delta^{\alpha+1/2}.
\end{eqnarray*}
The second integral, when  $0<\alpha< 1/2$: 
\begin{eqnarray*}
\left|\int^s_{\delta \wedge s} \frac{ \partial^{n-1}_s f(u, s + \delta) - \partial^{n-1}_s f(u, s)}{f(u, s + \delta)^3}\, du \right| &\leq & \int^s_{ \delta\wedge s} \frac{c u^\alpha \delta} {u^{3/2}}\, du\\
&\leq  & \frac{c\delta}{1 -2 \alpha} (\delta^{\alpha - 1/2} - s^{\alpha - 1/2}) \\
& \leq & \frac{c}{1 - 2\alpha} \delta^{\alpha + 1/2}.
\end{eqnarray*}
In the case $\alpha = 1/2$, the second integral is bounded by
$$
\int^s_{\delta \wedge s} c\delta u^{-1}\, du = c\delta \log \frac{s}{s \wedge \delta} = c\delta \log^+ \frac{s}{\delta}.$$
The third integral: 
\begin{eqnarray*}
\left| \int^s_0 \frac{\partial^{n-1}_s f(u,s) (f(u,s)^3 - f(u, s + \delta)^3)}{f(u, s)^3 f(u, s + \delta)^3)}\, du\right| & \leq & \int^s_0  \frac{c u(u\delta u)}{u^3}\,du \\
& = & c\delta s\leq c \delta.
\end{eqnarray*}
The last integral:
\begin{eqnarray*}
\left|\int^{s + \delta}_ s \frac{\partial^{n-1}_s f(u, s + \delta)}{f(u, s + \delta)^3}\,du \right| \leq \int^{s + \delta}_s \frac{c u }{u^{3/2}} \, du &= &c (\sqrt{s+\delta} - \sqrt{s})\\
& = &\frac{c\delta}{\sqrt{s + \delta}+ \sqrt{s}}  \leq \frac{c}{\sqrt{s_0}} \delta.
\end{eqnarray*}

To finish the proof, it remains to show \eqref{LambdaStarEstimate}.  Set $\alpha = 1/2$ and 
 write
$$ I_{n - 1}(s + \delta) + I_{n - 1}(s - \delta)- 2 I_{n - 1} (s) = 
 \left[ I_{n - 1}(s + \delta) - I_{n - 1} (s) \right] - \left[ I_{n - 1} (s) - I_{n - 1}(s - \delta) \right].$$
As  with
$ I_{n - 1}(s + \delta) -I_{n-1}(s)$ above,  we can decompose  $I_{n-1}(s) - I_{n - 1}(s - \delta)$ 
 into the sum of four integrals.  In both cases, the first, third and fourth integrals yield adequate bounds.   When $\delta \geq s-\delta$, the second integral is also adequately controlled.  Thus, we assume $\delta < s-\delta$ and we only need to control the difference of the second integrals:
 \begin{equation*}
  \int^s_{\delta} \frac{ \partial^{n-1}_s f(u, s + \delta) - \partial^{n-1}_s f(u, s)}{f(u, s + \delta)^3}\, du  
    - \int^{s-\delta}_{\delta } \frac{ \partial^{n-1}_s f(u, s ) - \partial^{n-1}_s f(u, s-\delta)}{f(u, s )^3}\, du.
\end{equation*}
We can decompose this into the sum $J_1 + J_2 + J_3$ where
\begin{align*}    
    J_1 &= \int^{s-\delta}_{\delta } \frac{\left( f(u, s)^3 - f(u, s+\delta)^3 \right) \left(  \partial^{n-1}_s f(u, s+\delta) -  \partial^{n-1}_s f(u, s) \right)}{f(u, s + \delta)^3 f(u, s)^3} \, du \\
    J_2 &= \int^{s-\delta}_{\delta } \frac{  \partial^{n-1}_s f(u, s+ \delta ) + \partial^{n-1}_s f(u, s - \delta ) - 2  \partial^{n-1}_s f(u, s ) }{ f(u, s)^3} \, du\\
     J_3 &= \int^{s}_{s-\delta} \frac{ \partial^{n-1}_s f(u, s + \delta) - \partial^{n-1}_s f(u, s)}{f(u, s + \delta)^3}\, du.  
\end{align*}  
Then
$$ |J_1| \leq \int^{s-\delta}_{\delta } \frac{c(u\delta u )( u \sqrt{\delta}) }{u^3}\, du \leq c \delta^{3/2},$$
and
$$ |J_3| \leq \int_{s-\delta}^{s } c \delta u^{-1} \, du = c \delta \log \frac{s}{s-\delta} \leq c \delta \log \frac{T}{s_0}.$$ 
Since
\begin{align*}
 &\Big| \left[ \partial^{n-1}_s f(u, s+ \delta ) - \partial^{n-1}_s f(u, s) \right]
        - \left[ \partial^{n-1}_s f(u, s) - \partial^{n-1}_s f(u, s- \delta ) \right] \Big| \\
        &\;\;\;\;= \left| \int_s^{s+\delta} \partial^{n}_s f(u, r ) - \partial^{n}_s f(u, r-\delta ) \, dr \right| \\
        &\;\;\;\;\leq \int_s^{s+\delta} c \sqrt{\delta} \, dr \leq c \delta^{3/2}, 
\end{align*}
then 
$$|J_2| \leq \int^{s-\delta}_{\delta } \frac{ c\delta^{3/2}}{u^{3/2}} \, du \leq c \delta.$$
This establishes \eqref{LambdaStarEstimate} and completes the lemma.
\end{proof}

%-----------------------------------------------------------------------
\begin{lemma} Suppose $\lbda\in C^{n,\alpha}([0,T];M)$ with $n\geq 2$ and $\alpha\in (\frac{1}{2},1]$. 
Then $I_{n} \in C^{\alpha - 1/2} [ s_0, T] $, and
there exists $c=c(M,T,n)$ such that for all $0\leq s \leq s+ \delta\leq T$
$$|I_n(s + \delta) - I_n(s) | \leq \frac{c}{2 \alpha - 1} \delta^{\alpha - 1/2}.$$
\end{lemma}

\begin{proof}
We proceed in a manner similar to the previous proof.
\begin{eqnarray*}
I_n(s + \delta) - I_n(s) &= & \int^{\delta \wedge s}_0 \frac{ \partial^n_s f(u, s + \delta) - \partial^n_s f(u, s)}{f(u, s + \delta)^3}\,du\\
&+   &\int^s_{\delta \wedge s} \frac{ \partial^n_s f(u, s + \delta) - \partial^n_s f(u, s)}{f(u, s + \delta)^3}\,du\\
&   + &    \int^s_0 \frac{\partial^n_s f(u,s) (f(u,s)^3 - f(u, s + \delta)^3)}{f(u, s)^3 f(u, s + \delta)^3}\, du\\
&  + &\int^{s + \delta}_ s \frac{\partial^n_s f(u, s + \delta)}{f(u, s + \delta)^3}\, du.
\end{eqnarray*}
The first integral:
\begin{eqnarray*}
\left|\int^{\delta \wedge s}_0 \frac{ \partial^n_s f(u, s + \delta) - \partial^n_s f(u, s)}{f(u, s + \delta)^3}\,du\right|&  \leq & \int^{\delta \wedge s}_0 \frac{c\min(u^\alpha, \delta^\alpha)}{u^{3/2}}\, du\\
& \leq &c\int^{\delta \wedge s}_0 u^{\alpha - 3/2} \, du \leq \frac{c}{2\alpha - 1} \delta^{\alpha - 1/2}.
\end{eqnarray*}
The second integral:
\begin{eqnarray*}
\left|    \int^s_{\delta \wedge s} \frac{ \partial^n_s f(u, s + \delta) - \partial^n_s f(u, s) }{ f(u, s + \delta)^3}\,du \right| &  \leq   &  \int^s_{s \wedge \delta} \frac{c\min (u^\alpha, \delta^\alpha)}{u^{3/2}}\, du\\
&\leq & \int^s_{s \wedge \delta} \frac{c\delta^\alpha}{u^{3/2}}\, du \leq c\delta^{\alpha}(\delta^{-1/2}-s^{-1/2})\leq c \delta^{\alpha - 1/2}.
\end{eqnarray*}
The third integral:
\begin{eqnarray*}
\left|\int^s_0 \frac{\partial^n_s f(u,s) (f(u,s)^3 - f(u, s + \delta)^3)}{f(u, s)^3 f(u, s + \delta)^3}\, du \right| & \leq & \int^s_0 cu^\alpha \frac{u^2\delta}{u^3}\, du \\
&= &\int^s_0 c\delta u^{\alpha - 1}\,  du = \frac{c\delta}{\alpha}s^\alpha\leq c\delta^{\alpha - 1/2}.
\end{eqnarray*}
The last integral:
\begin{eqnarray*}
\left| \int^{s + \delta}_ s \frac{\partial^n_s f(u, s + \delta)}{f(u, s + \delta)^3} \, du \right| & \leq & \int^{s + \delta}_s \frac{cu^\alpha}{u^{3/2}}\, du =  \frac{c}{2\alpha - 1} ((s + \delta)^{\alpha - 1/2} - s^{\alpha - 1/2})\\
&\leq & \frac{c}{2\alpha - 1} \delta^{\alpha - 1/2}.
\end{eqnarray*}
\end{proof}

%---------------------------------------------------------------------------------------------------------------%

%----------------------------------------------------------------------------------
%       %       ANALYTICITY OF THE CURVE

\section{Real analyticity of $\g$}
\label{sec: analytic}
In this section we prove Theorem \ref{t:analytic}. There exists $\delta>0$ such that $\lbda$ can be extended (complex) analytically to $E=\{z\in\mb{C}: d(z,[0,T])\leq \delta\}$. 
Notice that $f(s,s) = \g(0,s)= \g(s)-\lbda(0)$ and $f(u,s,\e)$ converges uniformly to $f(u,s)$ on $D=\{(u,s):0<u\leq s, 0<s \leq T\}$. 
So it suffices to show that $f(u, s, \e)$ can be extended analytically in the same neighborhood of $D$ (in $\mb{C}\times \mb{C}$) for all $\e$.
Recall that $G(z,u,s) = \frac{-2}{z} + \lbda'(s - u)$ is analytic in $(z,u,s)$, hence by the dependence of solutions of ODE on parameters (see \cite[Theorem 8.1]{CL}) the function $f( \cdot,s,\e)$ in (\ref{e: ODE}) exists and is analytic in a neighborhood of $u=0$ for each $\e\in (0,1]$ and $s\in E$. 
The main difficulty is to show this neighborhood is the same for all $\e$ and $s$. 

The outline of this section is as follows:
First we show in Lemma \ref{l: complex} that the equation (\ref{e: ODE}) still has solution when $s$ is in the domain $$E_1 = \{t: 0< \R t<T + \delta_1, |\I t|<\delta_1 \}$$ with $\delta_1$ small enough and not depending on $\e$.
Then in Lemma \ref{l: extension} we  show that one can take complex u-derivatives in (\ref{e: ODE}), which means the solutions are extended analytically. 
Finally by \cite[Theorem 8.3]{CL} the solutions are analytic in $(u,s)$ on the same domain for all $\e$. 

% Existence of solution when s is complex
Let $M$ be an upper bound for the sup-norms of $\lbda'$ and $\lbda''$ on $E$. As a first step, we will show the following:

\begin{lemma} \label{l: complex}
 There exists $\delta_1 \in (0, \delta) $ depending on $\delta, M$ and $T$ such that for every $s\in E_1$ and $\e\in (0,1]$, the solution to the equation 
\begin{eqnarray*}
\partial_u f(u, s, \e) &= & \frac{-2}{f(u, s, \e)} + \lbda'(s - u),~~~~ u\geq 0,\\
f(0, s, \e) &= & i\e,
\end{eqnarray*}
exists uniquely for $u \in [0, \R s + \delta_1]$. Moreover,
$$
 \max(\sqrt{2u},\frac{\e}{2})  \leq  \I f(u, s, \e) \mbox{ for } 0 \leq u \leq \R s + \delta_1. 
$$
\end{lemma}

\begin{proof}
The solution $f(u,s, \e)$ exists on a neighborhood of $u=0$, and it continues to exists as long as it stays above the real line. 
The uniqueness of this solution comes from standard ODE techniques. 
To establish the results of the lemma, we will compare $f(u, s, \e)$ to $f(u,s_0,\e)$ where $s_0 = \R s$ and
 \begin{eqnarray*}
\partial_u f(u, s_0, \e) &= & \frac{-2}{f(u, s_0, \e)} + \lbda'(s_0 - u),~~~~ u\geq 0,\\
f(0, s_0, \e) &= & i\e.
\end{eqnarray*}
It follows from Lemma \ref{l: upward ODE} $(i,ii)$ that
\begin{eqnarray*}
\sqrt{3u + \e^2} & \leq & \I f(u, s_0, \e) \\
\mbox{and }\,\,\,\,\,\,|\R f(u, s_0, \e)| &\leq & \sqrt{u}~~~~~~~~~  \mbox{ for } 0\leq u \leq s_0 + \delta_1,
\end{eqnarray*}
where $\delta_1 < \delta$ will be specified momentarily.
By following the same argument in Lemma \ref{l: lambda in C^1},
we get a bound for the difference of $f(u,s,\e)$ and $f(u, s_0,\e)$:
$$|f(u, s, \e) - f(u, s_0, \e)| \leq C M u |s - s_0| \leq C M u \delta_1$$
whenever $0\leq u \leq S$ with 
$$S = \inf\{ 0\leq v \leq u_0+\delta_1: \I f(v, s, \e) < \frac{\e}{3} \mbox { or } \frac{|\R f(v, s, \e)|}{\I f(v, s, \e)} > C_1\}, $$
where $C_1$ is a constant in $(0,1)$ and close to $1$. 
It follows that 
$$\I f(u, s, \e) \geq \I f(u, s_0, \e) - C M u \delta_1 \geq \sqrt{3u+\e^2} - CMu\delta_1,$$
and 
$$ |\R f(u, s, \e)| \leq |\R f(u, s_0, \e)| + C M u \delta_1 \leq \sqrt{u} + C M u \delta_1.$$
By choosing $\delta_1$ small enough, $\I f(u, s, \e) \geq \max(\sqrt{2u},\e/2)$ and 
$$\frac{|\R f(u, s, \e)|}{\I f(u, s, \e)}  < C_1$$
for all $0\leq u \leq S$. It follows that $S=u_0 + \delta_1$ and the lemma follows.
\end{proof}

Now we will show that

\begin{lemma} \label{l: extension} For every $\e\in (0,1]$, $s\in E_1$ and $0< \tilde{u}<\R s + \delta_1$, there exist $r=r(\tilde{u}, M, \delta, T) \in (0, \delta - \delta_1)$ and an analytic extension of $f(\cdot, s, \e)$ on  $B_{\tilde{u}} =\{ z\in \mb{C}: |z - \tilde{u}| < r\}$ such that
$$\partial_u f(u, s, \e) = \frac{-2}{ f(u, s, \e)} + \lbda'(s - u).$$
\end{lemma}

\begin{proof}
We will use the Picard iteration to show that the equation
\begin{eqnarray}
 \label{e: g}  g'(u)   &    =          &        -\frac{2} {g(u)} + \lbda'( s - u),\\
\nonumber   g(\tilde{u})      & =   & f( \tilde{u}, s, \e) 
\end{eqnarray}
has a solution on $B_{\tilde{u}} =\{ z\in \mb{C}: |z - \tilde{u}| < r\}$, where  $r$ will be specified later. 
Indeed for $|u - \tilde{u}| < r$   define
$g_0(u) = f(\tilde{u}, s, \e)$ and
$$g_{n+1}(u) = f(\tilde{u}, s, \e) + \int^u_{  \tilde{u}} \frac{-2}{g_n(v)} + \lbda'(s-v) \, dv.$$
We will show by induction on $n$ that $g_n$ is well-defined and analytic in  $B_{\tilde{u}}$ and
$$\I g_n(u) \geq \sqrt{\tilde{u}}.$$
The base case $n=0$ is clear because of Lemma \ref{l: complex}.
Suppose the claim holds for $n$.  
The function $g_{n+1}$ is well-defined and analytic in $B_{\tilde{u}}$ since $ \frac{1}{g_n}$ is analytic in a simply connected domain. Now
\begin{eqnarray*}
\I g_{n+1}(u) & \geq &\I f(\tilde{u}, s, \e) - |u-\tilde u|\max_{ v\in B_{\tilde{u}}} \left( \frac{2}{|g_n(v)|} + |\lbda'(s-v)|\right) \\
 & \geq & \sqrt{2\tilde{u}} - r (\frac{2}{\sqrt{\tilde{u}}} + M).
 \end{eqnarray*}
 The claim holds for $n+1$ by choosing $r$ small enough depending on $\tilde{u}, M$ and $T$. We also require that $r$ is small enough so that $2r/\tilde{u}<1$.  Then the sequence $g_n$ converges uniformly in $B_{\tilde{u}}$ since
 %-----
 \begin{eqnarray*}
 |g_{n+1}(u) - g_n(u)| & \leq & |u - \tilde{u}| \max_{ v\in B_{\tilde{u}}}  \frac{2 |g_n(v) - g_{n-1}(v)|}{|g_n(v) g_{n-1}(v)|}\\
 %---------%
 & \leq & \frac{2r} { \tilde{u}} ||g_n - g_{n-1}||_{B_{\tilde{u}},\infty}.
 \end{eqnarray*}

\noindent
Let $g$ be the limit. Then this function is analytic and satisfies the differential equation (\ref{e: g}). In particular $g(u)$ and $f(u, \tilde{u}, \e)$ solve same initial value problem. Hence they are equal when $u$ is real. In order words, $f(\cdot, s, \e)$ is extended analytically on $B_{\tilde{u}}$.
 \end{proof}
 
 \emph{Proof of Theorem \ref{t:analytic}.}
By \cite[Theorem 8.3]{CL}, for every $\e\in (0,1]$ the function $f(u, s, \e)$ is analytic in the domain $\{(u,s): s\in E_1, u\in B_{\tilde{u}} \mbox{ for some } \tilde{u}\in (0, \R s + \delta_1) \}$. It follows that $f(u,s)$ is also analytic in the same domain which contains $\{(s,s): 0<s \leq T\}$. Hence $f(s,s)$ and $\g(s)$ is real analytic on $(0,T]$. \qed

%
%%
%% SMOOTHNESS OF GAMMA(s^2)
%%
%%
%%
\section{Behavior of $\gamma$ at $s=0$}\label{sec:behavior at 0}

In this section we analyze the behavior of $\gamma$ at its base, proving
Theorem \ref{t: main theorem 2} and Theorem \ref{t:series at 0}.

\subsection{Smoothness of $\g(s^2)$ at $s=0$}

We may extend $\lbda$ smoothly on $(-\delta, T)$ by the concatenation property of the Loewner equation.
Thus, it suffices to show that for fixed $t_0 \in (0, T)$, the curve $\g_0(s^2) = g_{t_0} (\g(s^2 + t_0))$ is smooth at $s=0$ provided $\g$ is smooth on $(0,T)$. 
The idea, illustrated in Figure \ref{at0}, is as follows. Let $U$ be the intersection of $\mb{H}$ and a small disk centered at $\lbda(0)$ and let $V=g_{t_0}^{-1}(U)$. Define an analytic branch $\phi$ of $\sqrt{z-\g(t_0)}$ in a neighborhood of $\g(t_0)$ such that the branch cut is $\g(0,t_0]$. Let $W=\phi(V)$. All we need to check is that for small $\e>0$ the images under $\phi$ of $\g((t_0-\e,t_0])$ and $\g(t_0+s^2), 0\leq s^2\leq \e$, are smooth. Finally the smoothness of $\g_0(s^2)$ follows immediately from the Schwarz reflection principle through $E=\phi(\g((t_0-\e,t]))$ (in the case $\g$ is analytic) 
or Kellogg-Warschawski  theorem (in the case $\g$ is $C^{n, \alpha}$) for the map $\phi\circ g_{t_0}^{-1}$ from $U$ to $W$.

\begin{figure}[h]

\centering
\includegraphics[width=5in]{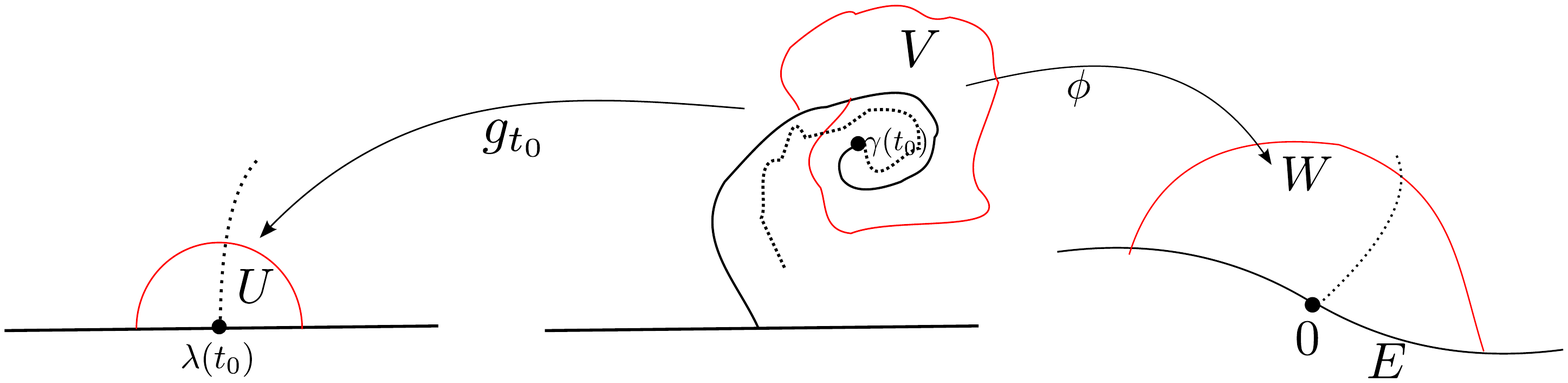}
\vspace{-1.5in}
\caption{ Illustration for the proof of Theorem \ref{t: main theorem 2}}\label{at0}

\end{figure}

\begin{proof}[Proof of Theorem \ref{t: main theorem 2} when $\l$ is analytic]
It follows from (\ref{e: gamma''}) that $\gamma'(t) \neq 0$ for all $t$. Thus, there exists an (real) analytic function $h$ on $(-\sqrt{\e},\sqrt{\e})$ such that
$$\frac{\g(t_0+s) - \g(t_0)}{s} = h(s)^2 \mbox{ for all } s\in (-\sqrt{\e},\sqrt{\e}) \backslash \{0\}. $$

Let $\phi_1(s) = is h(-s^2)$ and $\phi_2(s) = sh(s^2)$. We see that these two functions are analytic and one-to-one. Moreover,
$$\phi_1(s)^2 = \g(t_0 - s^2) - \g(t_0) \mbox{ and }$$
$$\phi_2(s)^2 = \g( t_0 + s^2) - \g(t_0).$$ 
Therefore the boundary $E$ of $W$, which is parametrized by $\phi_1(s)$ near 0, and $\phi(\g(t_0 + s^2))$ are analytic. 
Since the latter map is the image of $\g_0(s^2)$ under $\phi \circ g_{t_0}^{-1}$, it follows from the Schwarz reflection principle that $\g_0(s^2)$ is analytic at $0$.
\end{proof}

\begin{proof}[Proof of Theorem \ref{t: main theorem 2} when $\l$ is $C^{\beta}$]
By Theorem \ref{t: main theorem 2}, $\g \in C^{n,\a}(0,T]$ for appropriate $\a \in (0, 1)$.
 It is not obvious that the function $h$ in the previous case is $C^{n,\a}$. 
Indeed one can find an example of function $\g \in C^{n,\a}$ but $h$ is not.  Now let
$$H(s) = \frac{\g(t_0 + s) - \g(t_0)}{s} \mbox{ for } s \in (-\sqrt{\e}, \sqrt{\e})\backslash\{0\}, \mbox{ and } H(0)=\g'(t_0).$$
We claim that $H\in C^{n-1,\a} (-\sqrt{\e},\sqrt{\e})$. Indeed
$$H^{(n)}(s) = \frac{n!}{s^{n+1}} \sum^{n}_{k=0} \frac{(-1)^k}{k!} s^k \g^{(k)}( t_0+s)
- \frac{(-1)^n n!}{s^{n+1}} \g(t_0)
 \mbox{ for } s\neq 0.$$
Apply Proposition \ref{Cn_alpha} for functions $\g,\g',\cdots, \g^{(n)}$ to get $|H^{(n)}(s)| \leq cs^{\a - 1}$ which implies the claim.

Since $\inf_{s\in (-\sqrt{\e},\sqrt{\e})} |H(s)|>0$, it follows from the claim that the function $s \mapsto \sqrt{H(-s^2)} $ is $C^{n-1,\a}(-\sqrt{\e},\sqrt{\e})$
 for any well-defined square-root function.
Let $\phi_1(s)$ be a parametrization near 0 of $E$ such that $\phi_1(s)^2 =\g(t_0 - s^2) - \g(t_0)$ and $\phi_1(s) = s\sqrt{H(-s^2)}$ for $s \in (-\sqrt{\e}, \sqrt{\e})$. 
Since $\phi_1'(s) = \dfrac{\g'(t_0 - s^2)}{\sqrt{H(-s^2)}}$, the function $\phi_1$ is $C^{n,\a}(-\sqrt{\e},\sqrt{\e})$. 
The same argument shows that the function $\phi(\g(t_0 + s^2))$ is $C^{n,\a}[0,\sqrt\e)$. 
Combined with the last two statements,
the  Kellogg-Warschawski theorem \cite[Theorem  3.6]{Pommerenke} 
implies that the function $\g_0(s^2)$ is $C^{n,\a}[0,\sqrt\e)$.
\end{proof}

{\bf Remark.} The proof also shows that  if $\lbda\in C^{n,\alpha}([0,T];M)$ then $\G\in C^{n,\alpha+1/2}([0,T]; c)$ 
with $c=c(T, M, n, \alpha)$.

%%---------------------------------------------------------------------------------------------------------------%
\subsection{Expansion of $\g$ at $s=0$}\label{sec: comparison curves}

The goal of this section is to prove Theorem \ref{t:series at 0}, which illuminates why the $s^2$ parametrization is a natural parametrization at the base of a Loewner curve $\g$.
To accomplish this, we create a comparison curve $\tilde \g$ that closely approximates $\g$ near its base and is ``nice" at $s=0$ (that is, $\tilde \G(s) =\tilde \g(s^2)$ is smooth at $t=0$.)  
The properties of the comparison curve are summarized in Proposition \ref{prop} below.

Assume   $\g$ is generated by $\l \in C^{n,\a}[0,T]$.  We define $\tilde{\g}$ as a perturbation of a vertical slit, as done in Section 4.6 of \cite{L}.
Set 
$$\phi(z) 
 = z + \sum_{m=2}^{4n+1} \frac{b_m}{2^m} z^m,$$
 which is conformal on a neighborhood of the origin.
The real-valued coefficients $b_m$ will depend on $\l^{(k)}(0)$ as we will describe later.
Then define
\begin{align*}
\tilde{\g}(t) = \phi(2i\sqrt{t}) 
                        &=   2i\sqrt{t} + \sum_{m=2}^{4n+1} i^{m} b_m t^{m/2} \\
                        &=  2i\sqrt{t} - b_2 t - i \,b_3 t^{3/2} + b_4 t^2 + \cdots + i \,b_{4n+1} t^{2n+1/2}.
\end{align*}
Let $g_t: \H \setminus [0,2i\sqrt{t}] \to \H$ and $\tilde{g}_t:\H \setminus \tilde\g[0,t] \to \H$ be conformal maps with the hydrodynamic normalization at infinity.  Then we set $\phi_t = \tilde{g}_t \circ \phi \circ g_t^{-1}$
and $\tilde{\l}(t) = \phi_t(0)$, as illustrated in Figure \ref{comparison curve figure}.
In this form, $\tilde{\g}$ and $\tilde{\l}$  are not parametrized by halfplane capacity.  We will need to reparametrize by $t=t(s)$, which satisfies $t(0)=0$ and $\frac{dt}{ds} = \phi_t'(0)^{-2}$.  
Note in particular that $\frac{dt}{ds}\big|_{s=0} = 1$.

% FIGURE

\begin{figure}
\begin{tikzpicture}

\draw[thick] (0,0) -- (0,2);
\draw[fill] (0,2) circle [radius=0.05];
\draw[thick,dashed] (0,2) -- (0,3.2);
\node[right] at (0,2) {$2i\sqrt t$};
\draw (-3,0) -- (3,0);

\draw[->] (3,2) to [out=45,in=135] (5,2);
\node[above] at (4,2.5) {$\phi$};
\draw[->] (3,-3) to [out=45,in=135] (5,-3);
\node[above] at (4,-2.5) {$\phi_t$};

\draw[->] (-3,-1) to [out=225,in=135] (-3,-3);
\node[left] at (-3.5,-2) {$g_t$};
\draw[->] (11,-1) to [out=315,in=45] (11,-3);
\node[right] at (11.5,-2) {$\tilde g_t$};

\draw[thick] (8,0) to [out=90, in=270] (8.5,1) to [out=90,in=290] (7.8,1.9);
\draw[fill] (7.8,1.9) circle [radius=0.05];
\draw[dashed] (7.8,1.9) to [out=100,in=250] (8.1,3);
\node[right] at (7.8,1.9) {$\tilde\gamma(t)$};
\draw (5,0) -- (11,0);

\draw [dashed] (0,-5) -- (0,-3.8);
\draw[very thick] (-2,-5) -- (2,-5);
\draw[fill] (0,-5) circle [radius=0.05];
\node[below] at (0,-5) {$0$};
\draw (-3,-5) -- (3,-5);

\draw [dashed] (8.2,-5) to [out=90,in=210] (9,-4.2);
\draw [very thick] (7.5,-5) -- (10.5,-5);
\draw[fill] (8.2,-5) circle [radius=0.05];
\node[below] at (8.2,-5) {$\tilde\lambda(t)$};
\draw (5,-5) -- (11,-5);

\end{tikzpicture}
\caption{The conformal maps $\phi, g_t, \tilde g_t, \phi_t,$ the comparison curve $\tilde \g$, and  $\tilde \l$. 
}\label{comparison curve figure}
\end{figure}

% LEMMA

\begin{lemma} \label{polylem}
Assume $\phi_t$, $\tilde \l$ and $t=t(s)$ are defined as above, and let $k \in \mathbb{N}$.
Then there exists $\tilde T >0$, there exist polynomials
	 $p_k(x_1, x_2, \cdots, x_{k+2}  ),$ 
	 $ q_k(x_1, x_2, \cdots, x_{2k}  )$ 
	and $r_k(x_1, x_2, \cdots, x_{2k-1}  )$, 
 and there exist nonzero constants $c_k, d_k, e_k$  so that for $t \in [0, \tilde T],$
\begin{align}  
 \label{polylem1}
  & \partial_t \phi^{(k)}_t(0) = c_k \,\phi^{(k+2)}_t(0)+
                   p_k\left( \phi'_t(0), \phi''_t(0), \cdots, \phi^{(k+1)}_t(0), \phi_t'(0)^{-1} \right), \\
 \label{polylem2}                  
  & \partial_s^{k} \tilde \l(t)   =  d_k \, \phi^{(2k)}_t(0) \cdot  \phi_t'(0)^{-2k} +
		q_k\left( \phi'_t(0), \phi''_t(0), \cdots, \phi^{(2k-1)}_t(0), \phi_t'(0)^{-1} \right), \text{ and}\\
 \label{polylem3}
  &\partial_s^{k} t = e_k \, \phi_t^{(2k-1)}(0) \cdot \phi_t'(0)^{-(2k+1)} + r_k\left( \phi'_t(0), \phi''_t(0), \cdots, \phi^{(2k-2)}_t(0),  \phi_t'(0)^{-1} \right).
\end{align}
Further $\tilde \l \in C^{\infty}[0, s(\tilde T)]$ under the halfplane-capacity parametrization.
\end{lemma}

% LEMMA PROOF

\begin{proof}

Write $ \phi_t(z) = \sum_{k=0}^{\infty} a_k z^k$, keeping in mind that $a_k$ depends on $t$.
Then from Proposition 4.40 in \cite{L},
\begin{align}
\nonumber
\partial_t \phi_t(z) &= 2 \left( \frac{ \phi_t'(0)^2}{\phi_t(z) - \phi_t(0)} - \frac{ \phi_t'(z)}{z} \right)\\ \label{uglysums}
	&= -2\, \frac{ \sum_{k=0}^{\infty} (a_1a_{k+2}+2 a_2 a_{k+1}+ \cdots + (k+2) a_{k+2} a_1)z^{k}}
			{\sum_{k=0}^{\infty} a_{k+1}z^k}.
\end{align}
Since $a_1=1$ when $t=0$, there exists a neighborhood $U$ of 0 and $\tilde T>0$ so that the denominator is nonzero for $z \in U$  and $t \leq \tilde T$. 
 Therefore $ \partial_t \phi^{(k)}_t(z)$  is defined for $(z,t) \in U \times [0,\tilde T]$. 
 Equation \eqref{polylem1} follows from  \eqref{uglysums} (with $c_k = -\frac{2(k+3)}{(k+2)(k+1)}$.)

 We verify \eqref{polylem2} inductively.  For the base case,
 $$\partial_s \tilde \l(t) = \partial_t \phi_t(0) \cdot \frac{dt}{ds} = -3\, \phi_t''(0) \cdot \phi_t'(0)^{-2}.$$
 Assume  \eqref{polylem2} holds for a fixed $k$.  Then
 \begin{align*}
  \partial_s^{k+1} \tilde \l(t) &= \partial_t \left(  d_k \, \phi^{(2k)}_t(0) \cdot  \phi_t'(0)^{-2k} +
		q_k\left( \phi'_t(0), \phi''_t(0), \cdots, \phi^{(2k-1)}_t(0), \phi_t'(0)^{-1} \right) \right)\cdot  \phi_t'(0)^{-2} \\
	& = d_k \,c_{2k} \, \phi_t^{(2k+2)}(0) \cdot \phi_t'(0)^{-2k-2} + q_{k+1}\left( \phi'_t(0), \phi''_t(0), \cdots, \phi^{(2k+1)}_t(0), \phi_t'(0)^{-1} \right).
 \end{align*}
 
 We also prove \eqref{polylem3} inductively. When $k=1$,
 $$ \frac{dt}{ds} = \phi_t'(0) \cdot \phi_t'(0)^{-3}.$$
 If \eqref{polylem3} holds for fixed $k$, then
 \begin{align*}
 \partial_s^{k+1} t &= \frac{d}{dt} \left( e_k \, \phi_t^{(2k-1)}(0) \cdot \phi_t'(0)^{-(2k+1)} + r_k\left( \phi'_t(0), 
 	\phi''_t(0), \cdots, \phi^{(2k-2)}_t(0),  \phi_t'(0)^{-1} \right) \right)\cdot \phi_t'(0)^{-2} \\
    &= e_k	\, c_{2k-1} \, \phi_t^{(2k+1)}(0) \cdot \phi_t'(0)^{-(2k+3)} + r_{k+1} \left( \phi'_t(0), \phi''_t(0), \cdots, \phi^{(2k)}_t(0),  \phi_t'(0)^{-1} \right).
\end{align*}

The last assertion follows from \eqref{polylem2}. 
\end{proof}

% DEFINE COEFFICIENTS

We are now ready to recursively define the coefficients of $\phi$.  
The coefficient $b_m$ will depend on $\l^{(k)}(0)$ for $k=1, \cdots, \lfloor \frac{m}{2} \rfloor \wedge n$. 
For even values of $m$, our choice of $b_m$ will ensure that $\partial_s^{k} \tilde{\l}(0) = \l^{(k)}(0)$ for $k \leq n$.
For odd values of $m$, we choose $b_m$ so that the $t$-parametrization of $\tilde \g$ is close to the halfplane-capacity parametrization.
\begin{itemize}
\item Set $b_2 = -\frac{2}{3}\l'(0)$.  
	Since $\partial_s \tilde{\l}(0)= -\frac{3}{2} b_2$,  this implies that $ \partial_s \tilde \l(0)= \l'(0)$.
	\smallskip
\item Set $\displaystyle b_3 = \frac{b_2^2}{8}$.
	This implies that  $ \frac{d^{2}t}{d s^{2}} \big|_{s=0}= 2 b_3 - b_2^2/4=0$. 
\smallskip
\item  Assume that  $b_2, b_3, \cdots, b_{2k-1}$ have been defined.  
Then by Lemma \ref{polylem}, 
    $$ \partial_s^{k} \tilde \l(0)   =  d_k \, \frac{(2k)!}{2^{2k}}\, b_{2k}  +
		q_k\left( 1,  \frac{1}{2}b_2, \cdots, \frac{ (2k-1)!}{2^{2k-1} }b_{2k-1}, 1 \right) .$$
If $k \leq n$, define $b_{2k}$ so that $\partial_s^{k} \tilde{\l}(0) = \l^{(k)}(0)$.  If $k>n$, we may define $b_{2k}$ however we like; for instance, we choose $b_{2k}$  so that $\partial_s^{k} \tilde{\l}(0) =0$.
\item  Assume that  $b_2, b_3, \cdots, b_{2k}$ have been defined.  
\smallskip
Then  by Lemma \ref{polylem},
  $$ \frac{d^{k+1}t}{d s^{k+1}}\bigg|_{s=0} = e_{k+1} \, \frac{(2k+1)!}{2^{2k+1}} \,  b_{2k+1} 
  		+ r_{k+1} \left( 1,  \frac{1}{2}b_2, \cdots, \frac{ (2k)!}{2^{2k} }b_{2k}, 1 \right).$$
  Define $b_{2k+1}$ so that this quantity is zero.
\end{itemize}

This construction ensures that 
$\partial_s^{k} \tilde{\l}(0) = \l^{(k)}(0)$ for $k \leq n$
and that
 $t = s + O(s^{2n+2})$.  
 The first fact, together with by Theorem 3.3 in \cite{C}, implies that 
$ |\g(s) - \tilde \g(t(s))| = O(s^{n+\a})$ for $s$ near 0. 
The second fact implies that under the halfplane-capacity parametrization $\tilde \g(t(s))$ will have the same coefficients as $\tilde \g(t)$ for the terms with exponents at most  $n+1/2$.
Together, this provides precise information about the  expansion of $\g(s)$ near $s=0$.
In summary, we have proved the following, which establishes Theorem \ref{t:series at 0}.

% SUMMARY PROPOSITION

\begin{proposition} \label{prop}
Assume that $\l \in C^{n,\alpha} [0,T]$ generates the curve $\g$.
Then there exists $\tilde \l \in C^{\infty} [0,S] $ that generates a (halfplane-capacity-parametrized) curve $\tilde \g \in C^{\infty}(0, S]$ with the following properties:
\begin{itemize}
\item $\l^{(k)}(0) = \tilde \l^{(k)}(0)$ for $1 \leq k \leq n$. 
\smallskip
\item $\tilde \G(s) = \tilde \g(s^2)$ is  in $C^{\infty}[0,\sqrt{S}]$.
\smallskip
\item $\tilde \G^{(m)}(0)$ depends on $\l^{(k)}(0)$ for $m \leq 2n+1$ and $k=1, \cdots, \lfloor \frac{m}{2} \rfloor$.
\smallskip
\item $|\g(s) - \tilde \g(s)| = O(s^{n+\alpha}).$
\end{itemize}
In particular near $s=0$, the curve $\g$ has the  form
\begin{equation*}
\g(s) =  
\begin{cases}
 2i\sqrt{s} + a_2 s + i \,a_3 s^{3/2} + a_4 s^2 + \cdots + a_{2n} s^{n} + O(s^{n+\a})
	&\mbox{if   } \a \leq 1/2 \\ 
 2i\sqrt{s} + a_2 s + i \,a_3 s^{3/2} + a_4 s^2 + \cdots +  a_{2n} s^{n} +i \,a_{2n+1} s^{n+1/2}   + O(s^{n+\a})
 &\mbox{if   } \a > 1/2 \\ 
 \end{cases}
\end{equation*}
where the real-valued coefficients $a_m$ depend on 
$\l^{(k)}(0)$ for $k=1, \cdots, \lfloor \frac{m}{2} \rfloor$.
\end{proposition}

We note the equations for the first few coefficients: 
\begin{align*}
a_2 &= \frac{2}{3} \l'(0) \\
a_3 &= -\frac{1}{18} \l'(0)^2 \\
a_4 &= \frac{4}{15} \l''(0) + \frac{1}{135} \l'(0)^3 \\
a_5 & = -\frac{1}{15} \l''(0) \l'(0) + \frac{1}{2160} \l'(0)^4 
\end{align*}
Coefficients $a_2, a_3, a_4$ were discovered in \cite{LR2} by comparison with specific example curves (such as those generated by $c\sqrt{\tau-t}$.)

Along with the tools developed in Sections \ref{Properties} and \ref{smoothness section}, Proposition \ref{prop} could be used to show that if $\G(s) = \g(s^2)$, then
$\G^{(k)}(0)$ exists and  equals $\tilde \G^{(k)}(0)$ for $k=1, \cdots, n+1$.

%---------------------------------------------------------------------------------------------------------------%
  %EXAMPLE SECTION
 
 \section{Examples} \label{sec:ex}
  
In this section we discuss two examples that illustrate the two special cases of Theorem \ref{t: quantitative}.  
The first special case is when the driving function is $C^{n+1/2}$.  Here the conclusion is weaker than we might initially expect: it is not necessarily true that $\g \in C^{n+1}$, but rather  $\g$ is in the larger space $\Lambda^n_*$ (which contains both $C^{n+1}$ and $C^{n,1}$.)
This case is illustrated in the first example where the driving function is $C^{3/2}$ and the associated curve is $C^{1,1}$ but not $C^{2}$.  
The second special case of   Theorem \ref{t: quantitative} is when the driving function is $ C^{n,1}$.  Here the conclusion is slightly stronger than might be initially expected: $\g \in C^{n+1, 1/2}.$
This is illustrated in the second example, where
 the driving function is $C^{0,1}$ but not $C^1$ and the associated curve is $C^{3/2}$.
 We describe the needed computational steps to verify these examples, but leave details for the reader.

\subsection{Example 1: $\lambda \in C^{3/2}$ and $\gamma \in C^{1,1}\backslash C^2$}

This example was communicated to us from Don Marshall.  

We will create $\gamma$ via a sequence of conformal maps, as pictured in Figure \ref{fig:rect-line}.
Let $f_1(z) = z+\frac{1}{z} + c\ln z$, and let $r_{1,2} = \frac{-c\pm \sqrt{c^2 + 4}}{2}$ be the finite critical points of $f_1$.
Define 
$$g(z)=\frac{c\pi}{f_1(z)-f_1(r_1)},$$
which is a conformal map from $\mb{H}$ onto the $C^{1,1}$ domain $\mb{C}\backslash((-\infty,0]\cup \text{ a circle arc})$.
Finally, set 
$$F(z) = i\sqrt{g(z)+1}.$$
The image of $\mb{H}$ under $F$ is a slit half-plane, and we let $\gamma$ be the resulting slit.

\begin{figure}[h!]
        		
            		\includegraphics[scale=0.3]{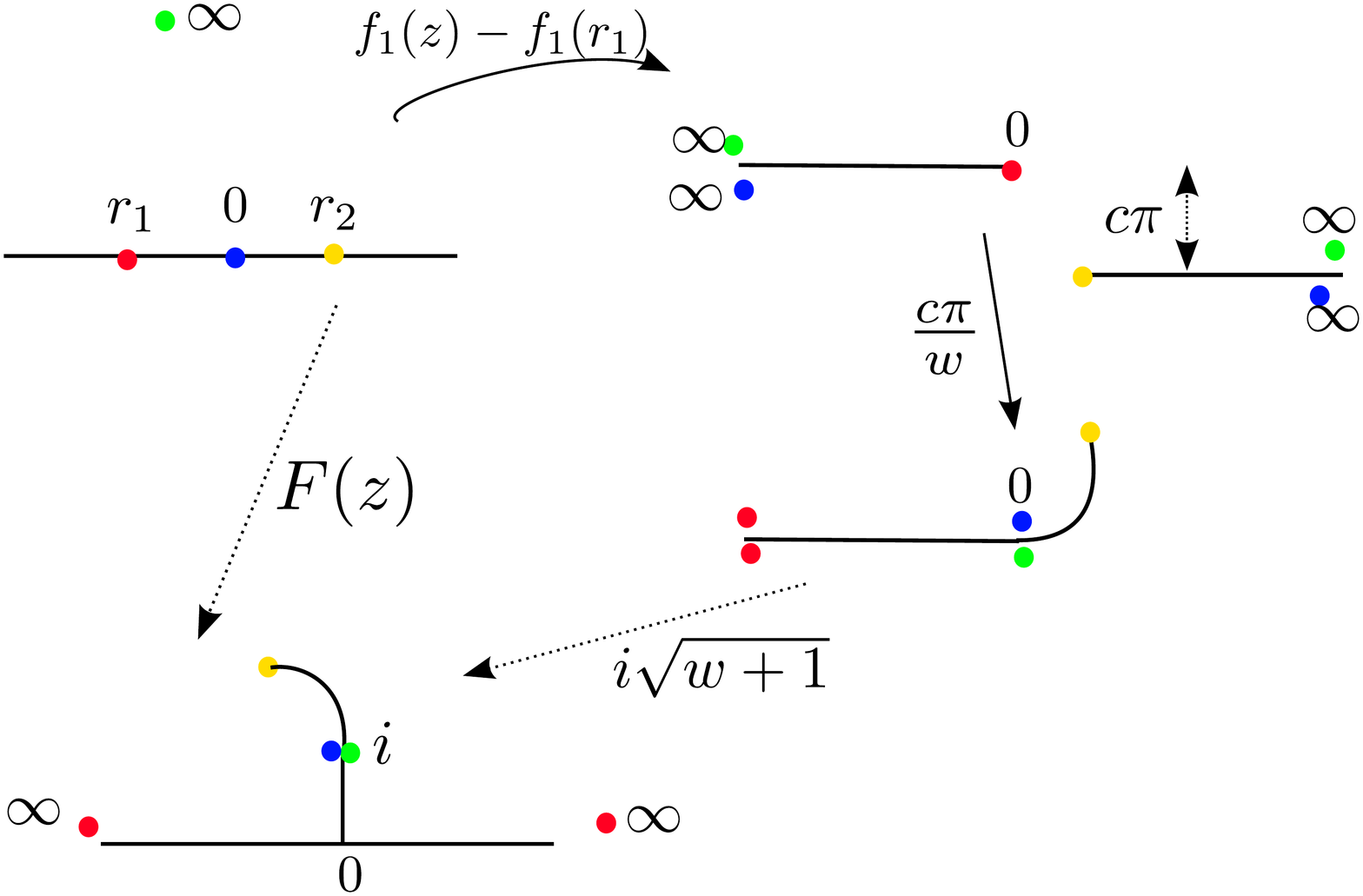}
        		\caption{Conformal maps used in the construction of $\g$ for Example 1.}\label{fig:rect-line}
 	\end{figure}

For $t \in [0,1/4]$, $\gamma(t) = 2i\sqrt{t}$ and $\lambda(t) \equiv 0$. 
To compute $\l$ and $\g$ for $t >1/4$, we will need to use the conformal maps,
 since $\g(t) = F(r_2)$ and $\l(t) = L^{-1}(r_2)$ for the automorphism $L$ of $\mb{H}$ 
with
\begin{equation}\label{normalization}
F(L(z))=z+ 0 + \frac{-2t}{z} + \cdots \; \text{ near infinity}.
\end{equation}

Since $L$ must send $\infty$ to $r_1$, 
\begin{equation*}
L(z) = r_1 + \frac{a}{z-b}
= r_1+\frac{a}{z} + \frac{ab}{z^2} +\frac{ab^2}{z^3}+\frac{ab^3}{z^4} + O(|z|^{-5})\text{ near infinity},
\end{equation*}
where $a<0$ and $b\in\mb{R}$.  
Using this and the Taylor series expansion of $f_1 - f_1(r_1)$ at $z=r_1$, one can compute that
$$ f_1(L(z))-f_1(r_1) = \frac{A}{z^2}+\frac{B}{z^3}+\frac{D}{z^4} + O(1/|z|^5) \text{ near infinity},$$
with $$A= \frac{a^2f_1^{(2)}(r_1)}{2},  \;\;  B= a^2b f^{(2)}(r_1)+ \frac{a^3 f^{(3)}(r_1)}{6},$$
  $$ \text{ and } D = \frac{3a^2b^2 f^{(2)}(r_1)}{2} + \frac{a^3 b f^{(3)}(r_1)}{2} + \frac{a^4 f^{(4)}(r_1)}{24}.$$
Thus near infinity,
\begin{align*}
 F(L(z)) 
  &= i\sqrt{\frac{c\pi}{A}z^2 -\frac{c\pi B}{A^2} z -\frac{c\pi D}{A^2} + \frac{c\pi B^2}{A^3} + 1 + O(1/|z|)} \\
  &= i\left(-i\sqrt{\frac{c\pi}{|A|}} z -i B \frac{\sqrt{c\pi}}{2|A|^{3/2}}+ O(1/|z|)\right)
 \end{align*}
Note that in choosing the appropriate branch for the square root, we used the fact that $A <0$.
In order to satisfy \eqref{normalization}, we must have
\begin{itemize}
\item $ \displaystyle A = -c\pi, \,$  or equivalently,  $\displaystyle a =\frac{r_1\sqrt{-2\pi cr_1}}{\sqrt{2-cr_1}}$, and
\item $\displaystyle B=0, \, $ or equivalently,  $\displaystyle b = \frac{(cr_1-3)\sqrt{-2\pi c r_1}}{3(2-cr_1)^{3/2}}.$
\end{itemize}
Using these two facts, we expand further and find that   at infinity,
$$F(L(z)) = z+0 - \frac{1}{2}\left(\frac{D}{A}+1\right) \frac{1}{z}+ O(1/|z|^2),$$
which implies that
$$4t = \frac{D}{A}+1= \frac{-\pi c r_1(c^2r_1^2-6c r_1+6)}{3(2-cr_1)^3} +1.$$

Next we compute $\l(t)$ for $t>1/4$:
$$\l(t) = L^{-1}(r_2) = b+\frac{a}{r_2-r_1}=\frac{-2\sqrt{2\pi}(-cr_1)^{3/2}}{3(2-cr_1)^{3/2}}.$$
Thus with $y=-cr_1$, we have 
$$t= \frac{1}{4} +\frac{\pi y(y^2+6y+6)}{12(2+y)^{3}}
     \;  \text{  and  } \; \l(t) = \frac{-2\sqrt{2\pi}y^{3/2}}{3(2+y)^{3/2}}.$$
So for $t>1/4$, 
$$\l'(t) = \frac{\frac{d\l}{dy} }{ \frac{dt}{dy}} = \frac{-2\sqrt{2}\sqrt{y}(2+y)^{3/2}}{\sqrt{\pi}(y+1)}.$$
Using this, one can show that for $s> t \geq 1/4$,
$$|\l'(s)-\l'(t)| \leq c \sqrt{y_s-y_t} \leq c' \sqrt{s-t},$$
proving that $\l \in C^{3/2}[0,T]$.  We also note that away from $t=1/4$, one can check that $\l(t)$ is $C^2$.

Lastly, for $t \geq 1/4$, $\g(t) = F(r_2)$.  Using this, one can determine computationally that with the halfplane-capacity parametrization, $\g'$ and $\g''$ exist on $[1/4, T]$ (by computing, for instance, $\g'(t) = \frac{dF(r_2)}{dc}/\frac{dt}{dc}$ and  $\g''= \frac{d\g'(t)}{dc}/\frac{dt}{dc}$).  Further,
$$ \lim_{t \searrow 1/4} \g'(t) = 2i =\lim_{t \nearrow 1/4} \g'(t), $$
but 
$$  \lim_{t \searrow 1/4} \g''(t)=-4i-16 \neq \lim_{t \nearrow 1/4} \g''(t) = -4i.$$
Therefore on the full interval $(0,T]$, $\g$ is $C^{1, 1}$ but not $C^2$.

\subsection{Example 2: $\lambda \in C^{0,1}$ and $\gamma \in C^{3/2}$}

Consider the driving function
\begin{equation*}
	\lbda(t) = \left\{ 
	\begin{array}{ccr}
	0 & \mbox{ for } & 0\leq t \leq \frac{1}{4}\\
	\frac{3}{2}-\frac{3}{2}\sqrt{1-8(t-1/4)} & \mbox{ for } & \frac{1}{4} \leq t < \frac{1}{4} + \frac{1}{10}\\
	\end{array}
	\right.  .
\end{equation*}
There exists $c>0$ so that 
$$|\lbda(t)-\lbda(s)|\leq c |t-s|$$
for all $s, t \in [0, 0.35]$, implying that $\l \in C^{0,1}$.
However, $\l$ is not in $C^1$ since $\lbda'$ is not continuous.

The driving function $ \frac{3}{2}-\frac{3}{2}\sqrt{1-8s}$, defined on $[0,\frac{1}{8}]$, generates the upper half-circle of radius $\frac{1}{2}$ centered at $\frac{1}{2}$.  
Let $\hat{\g}$ be the portion of this circle generated on the time interval $[0, \frac{1}{10}]$. 
 Then the curve $\g$ generated by $\l$ is the image of $[-1,1] \cup \hat\g$ by the map $S(z)=\sqrt{z^2-1}$.
See Figure \ref{Example2gamma}.
By Proposition 3.12 in \cite{MR07},
$\g \in C^{3/2}$ (and no better) under the arclength parametrization.  This is also true under the halfplane-capacity parametrization.  Note that  $\hat \g$ is  smooth on $(0, \frac{1}{10}]$ (because its driving function is smooth),
and  near $s=0$
$$\hat \g(s) = 2i\sqrt{s} + 4s -2is^{3/2} + O(s^2)$$
by Theorem \ref{t:series at 0}.
Thus $\g$ is piecewise smooth, and for $t \geq 1/4$
$$\g(t) = S(\hat \g(t-1/4))
    = i +2i(t-1/4)+8(t-1/4)^{3/2}+O((t-1/4)^2).$$
From this we can determine that $\g \in C^{3/2}(0,0.35]$ (and no better) under the halfplane-capacity parametrization.

\begin{figure}
\begin{tikzpicture}

\draw[thick] (0,0) arc[radius=1, start angle=180, end angle=45];
\node[right] at (-0.2,1) {$\hat{\gamma}$};
\draw (-2,0) -- (3,0);

\draw[->] (3,2) to [out=45,in=135] (5,2);
\node[above] at (4,2.5) {S};

\draw[thick] (7,0) to (7,1.8);
\draw[thick] (7,1.8) to [out=90,in=180] (7.25, 2);
\draw[thick] (7.25, 2) to [out=0,in=100] (7.8,1.4);
\node[right] at (7.7,1.9) {$\gamma$};
\draw (5,0) -- (10,0);

\end{tikzpicture}
\caption{The curve $\g$ for Example 2.
}\label{Example2gamma}
\end{figure}

\bibliographystyle{alpha}

\bibliography{LE_SLE}

\end{document}